\documentclass[smallextended,referee,envcountsect]{mysvjour3}
\smartqed
\usepackage{graphicx}

\usepackage{makeidx,epsfig}
\usepackage[margin=1.0in]{geometry}
\usepackage{amssymb,amsfonts}
\usepackage{amsmath}
\usepackage{amsbsy}
\usepackage{verbatim}
\usepackage[bottom]{footmisc}
\usepackage{algorithm}
\usepackage{algorithmicx}
\usepackage{algpseudocode}
\usepackage{colortbl,xcolor,pgfplots}
\usepackage{algorithm,algorithmicx,algpseudocode}

\newcommand{\OOmega}{\Omega\!\!\!\!\Omega}

\usepackage[latin1]{inputenc}
\usepackage[english]{babel}

\usepackage{makeidx}
\usepackage{capt-of}
\usepackage[T1]{fontenc}
\usepackage{amsfonts}
\usepackage{amscd}
\usepackage{amssymb}
\usepackage{tabularx}
\usepackage{amssymb,bezier,graphicx,}
\usepackage{times,amssymb}
\usepackage{graphicx}
\usepackage{subfigure}
\usepackage{color}
\usepackage{cancel}
\usepackage{fancybox}
\usepackage{fancyhdr}
\usepackage{hyperref}
\usepackage{tikz}
\usepackage{enumerate}
\usepackage{epstopdf}
\usepackage{array}
\usepackage{slashbox}
\usepackage{verbatim}
\usepackage{float}
\usepackage[colorinlistoftodos]{todonotes}
\usepackage{framed}
\usepackage{anysize}
\usepackage{pdfpages}
\usepackage{stackrel}
\usepackage{booktabs}
\usepackage[mathscr]{eucal}
\usepackage{bm}
\usepackage{multirow}
\usepackage{caption}
\usepackage{mathrsfs} 

\usetikzlibrary{shapes,arrows,trees,positioning,through}
\usetikzlibrary{calc}

 \allowdisplaybreaks

\newtheorem{Thm}{Theorem}[section]
\newtheorem{Lem}{Lemma}[section]

\newtheorem{Cor}{Corollary} 
\newtheorem{Def}{Definition}[section]

\begin{document}

\title{On location-allocation problems for dimensional facilities}


\author{Lina Mallozzi \and Justo Puerto \and Moisés Rodríguez-Madrena}

\institute{Lina Mallozzi \at
             University of Naples Federico II, Naples, Italy\\
              lina.mallozzi@unina.it
           \and
           Justo Puerto,  Corresponding author  \at
              IMUS, Universidad de Sevilla, Seville, Spain \\
              puerto@us.es
           \and
           Moisés Rodríguez-Madrena   \at
              IMUS, Universidad de Sevilla, Seville, Spain \\
              madrena@us.es
}


\maketitle

\begin{abstract}
This paper deals with a bilevel approach of the location-allocation problem with dimensional facilities. We present a general model that allows us to consider very general shapes of domains for the dimensional facilities and we prove the existence of optimal solutions under mild, natural assumptions. To achieve these results we borrow tools from optimal transport mass theory that allow us to give explicit solution structure of the considered lower level problem. We also provide a discretization approach that can approximate, up to any degree of accuracy, the optimal solution of the original problem. This discrete approximation can be optimally solved via a mixed-integer linear program. To address very large instance sizes we also provide a GRASP heuristic that performs rather well according to our experimental results. The paper also reports some experiments run on test data.
\end{abstract}


\section{Introduction}

Location-allocation problems are very important problems nowadays in the area of Operation Research and  Logistics: they consists of finding the placement of a number of servers and deciding the assignments of the existing demand in order to minimize some general objective function.
See for example \cite{LNG2016LS,MDP2017,NP05}.
Depending on the framework, the problem can be cast within the family  of continous non-convex or mixed-integer programming problems and in
some cases is closely related with the design of Voronoi partitions (\cite{Okabe}) in computational geometry.  These problems are important by themselves for their mathematical implications but also by their many applications to several important  areas such as territorial design, market share and hub-and-spoke design, voting districts, shape optimization,  etcetera (\cite{BLPartII,Diaz,Drezner,Kalcsics2015,LHmarketArea,surveyTOPJusto}).

Sometimes these servers can be identified with extended domains: in this case we will speak about  dimensional facilities.
Mathematically, a dimensional facility location problem corresponds to finding the best position of a geometrical  figure (\cite{NPR03,PR10}).
The resolution of the problem in this case must take care of the optimizing aspect of a certain utility function and also of the geometry of the facility.

In spite of their importance, to the best of our knowledge, the consideration of location-allocation problems with respect to
dimensional facilities has not been extensively considered in the literature. Some exception is the paper  \cite{optPartition}.

There is a number of papers in literature dealing with the so called location-allocation problem, i.e., a combination of the two tasks, where one asks for
the best positions of the servers together with the best partition of the demand. The location-allocation approach gives rise to a natural bilevel
optimization problem where in the first level the location decisions are made under the constraint that the allocation will be given as
a best reply function.  This bilevel problem is in general hard to solve. In the particular case where the facilities are dimensional it
becomes harder.  See for references Ch. 14 in \cite{LNG2016LS}, Ch. 5 in  \cite{MDP2017}.

Situations like these appear very often in Game Theory when two players compete in a hierarchical scheme and the model is usually called
Stackelberg game (or Leader/follower game).   In these bilevel problems  we have almost never an explicit expression of the solution for
the lower level problem to be considered and then be included to help in solving the upper level one.

Sometimes and under some suitable assumptions, the solution of the lower level problem (the so-called best reply) is obtained explicitly and this helps in the resolution of the upper level. This happens, for example, when we use optimal transport tools as done in \cite{CarMallozzi2018,MaPa}.

This theory started with the problem of moving a pile of sand into a hole of the same volume minimizing the transportation cost, formulated by Monge. Then, Kantorovich relaxed the problem providing a dual formulation. Recently these classical results have been used in a large number of application contexts as Transportation, Logistics, Physics, etc. (\cite{CuestaMatran,A,gal2017,Villani}).

By using optimal transport theory it is possible to obtain a structure of the solution of the lower level and then to  prove the existence of the solution of the bilevel model. Moreover, the obtained structure of the optimal partition, that optimizes the demand problem, is fundamental in order to develop some approximation results and some computational algorithms.

This paper generalizes previous  result in \cite{optPartition} since that paper only considered the lower level problem and with particular shapes for the dimensional facilities. Moreover, the contribution of this paper is threefold. First, we formulate the bilevel location-allocation problem for very general dimensional facilities and  prove, under suitable conditions, the existence of optimal solutions. Secondly, we give an approximation scheme to solve the problem, discretizing some of its elements, providing convergence results to the optimal solution of the original problem. Finally, we also develop an exact solution algorithm applicable to the discrete approximation scheme that reduces the problem to solve a mixed-integer linear problem. In addition,  we also propose a GRASP heuristic  that performs very-well experimentally in large size instances. The paper also reports our computational experiments with different test cases. For the sake of readability, we restrict ourselves to the 2-dimensional setting although most of the results in this paper extend further to finite dimension spaces.

The rest of the  paper is organized  as follows: in the first section the bilevel problem is presented and existence results of optimal solutions are obtained; in the second section a discretization scheme is defined and some convergence theorems are proved; in section three different solution approaches are compared: an exact mixed-integer linear programming model and a GRASP heuristic are tested and the reported are presented. The paper finishes with some conclusions and an outline for future research.

\section{A bilevel model and existence of optimal solutions}

\subsection{Bilevel approach}

We are given  $\Omega$, a Borel, compact subset of  $\mathbb{R}^2$, that represents a demand region. We assume that customers in $\Omega$ are distributed according to a  demand density $D\in \mathcal{L}^2(\Omega)$ that is  an absolutely continuous probability measure, where $D: \Omega \to \mathbb{R}$  is a nonnegative function with unit integral $\int_\Omega D(q) dq=1$, being $q=(x,y)\in \Omega$ and  $dq=dxdy$. The goal is to locate $\rho$ given compact sets $P_1,...,P_{\rho}$ ($\rho\in \mathbb{N}$) in $\Omega$, assuming that all of them are the closure of nonempty open connected sets, representing some service centers with dimensional extension. From now on, any set with these properties will be called a \emph{dimensional facility}.

For each $i\in \{1,...,\rho\}$, we consider that the location of the dimensional facility $P_i$ in the plane is determined by the location of a point $p_i=(px_i,py_i)\in P_i$ called its \emph{root point}:  we use the notation $P_i^{q_i}$ to refer to the dimensional facility $P_i$ when its root point $p_i$ is located (fixed) at point $q_i\in \mathbb{R}^2$. This means that the set $P_i^{\tilde{q}_i}$ is the set $P_i^{q_i}$ when we apply to it the translation induced by the vector $\overrightarrow{q_i \tilde{q}_i}\in \mathbb{R}^2$, for any $q_i,\tilde{q}_i\in \mathbb{R}^2$ (see Figure \ref{closedset}). In other words, the shape of the dimensional facility $P_i$ is the same for any possible location in the plane. Each dimensional facility is then determined within the region $\Omega$ locating its root point.

\begin{figure}[H]
\centering
\includegraphics[scale=0.3]{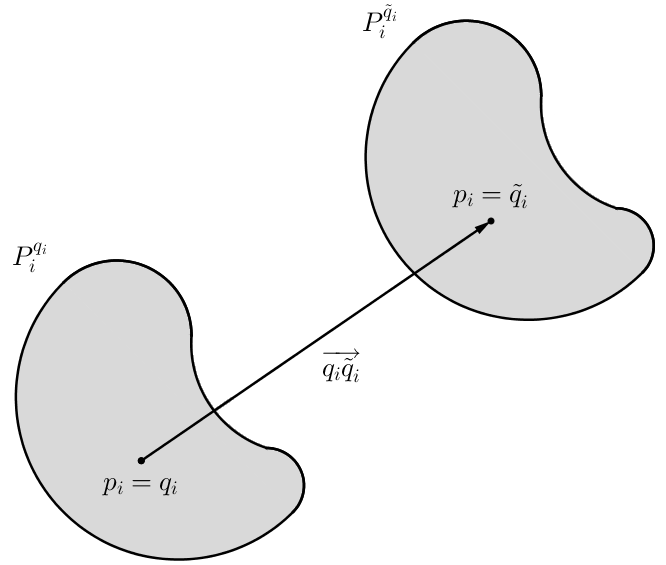}
\caption{Two possible locations in the plane for the dimensional facility $P_i$} \label{closedset}
\end{figure}

The problem considered in this paper is to locate $\rho$ dimensional facilities $P_1,...,P_{\rho}$ in $\Omega$ and also to find the partition (market share) $A_1,....,A_{\rho}$ satisfying that the dimensional facility $P_i$ serves the consumer demand in the region ${A}_i\subseteq \Omega$ optimizing a suitable criterion: we will find  a partition of the set $\Omega\setminus \{\text{int}(P_1)\cup...\cup \text{int}(P_{\rho})\}$, i.e., a finite family $(A_i)_{i=1}^{\rho}$ of pairwise disjoint Borel sets such that $\bigcup_{i=1}^{\rho} {A}_i= \Omega\setminus \{\text{int}(P_1)\cup...\cup \text{int}(P_{\rho})\}$ up to $D$-negligible sets.

We require that the location of the dimensional facilities $P_1,...,P_{\rho}$ in $\Omega$ must satisfy that the interior of the closed sets do not intersect and obviously that $P_i\subseteq \Omega$ for all $i\in\{1,...,\rho\}$.
A family of $\rho$ dimensional facilities that satisfy the above conditions will be called a \emph{suitable solution}. We also assume that there is a location of the dimensional facilities verifying the above conditions, i.e., the problem considered has at least one suitable solution.

In order to formally describe the set of suitable solutions for the dimensional facilities $P_1,...,P_{\rho}$, we introduce the following notation: let $\Omega_i$ denote the region of $\mathbb{R}^2$ in which locating $p_i$ makes  $P_i$ to be contained in $\Omega$, i.e.,
$$\Omega_i = \{q_i\in\mathbb{R}^2 : P_i^{q_i}\subseteq \Omega\},$$
for each $i\in\{1,...,\rho\}$. Obviously, $\Omega_i\subseteq \Omega$, for all $i\in\{1,...,\rho\}$.
Then, the set of suitable solutions is
$$\Gamma = \{(q_1,...,q_{\rho})\in \Omega_1\times ... \times \Omega_{\rho} : \text{int}(P_i^{q_i})\cap \text{int}(P_j^{q_j})=\emptyset, \forall i,j\in\{1,...,\rho\},i\not= j\}.$$ Clearly, $\Gamma\subseteq \Omega^{\rho}\subseteq \mathbb{R}^{2\rho}$. Recall that we are assuming that $\Gamma \not= \emptyset$.

We consider that the utility $u$ paid from a point $q\in \mathbb{R}^2$ with respect to the dimensional facility $P_i$ is given by a continuous function $u_i:\mathbb{R}^2\times \mathbb{R}^2 \to \mathbb{R}$ that depends on the considered point $q$ and the location $q_i\in \mathbb{R}^2$ of the dimensional facility $P_i$:
\begin{equation*} \label{dcons}
u (q, P_i^{q_i}) = u_i (q,q_i),
\end{equation*}
for each $i\in \{1,...,\rho\}$. To clarify the meaning of choosing the utility $u$ in this way, we indicate some interesting particular cases (among others) of $u$ and their interpretations:

\begin{enumerate}[$-$]

\item Service point case: this is the most intuitive situation. Here, the customer point $q\in \Omega$ has to   reach the service point in $P_i$ (or vice versa) to satisfy its demand. Assume that the role of the service point is played by the root point $p_i$ of $P_i$. Then, $u$ can be chosen as
$$ u (q, P_i^{q_i}) = f_i (\gamma_i (q-q_i)),$$
being $f_i:\mathbb{R} \to \mathbb{R}$ a continuous function and $\gamma_i: \mathbb{R}^2 \to \mathbb{R}$ a norm, and where we are considering a measure of the distance between $q$ and $q_i$ according to the norm $\gamma_i$. Note that, although in this case the utility does not depend on the shape of the dimensional facility $P_i$ but only on the location of $p_i$, the shape of the dimensional facilities still plays a role in the problem since it determines the set of suitable solutions $\Gamma$ and also some others aspects of the problem as we will see later.

\item Utility dependent on the shape of the facility: in this case the measure of the distance from the customer point $q\in \Omega$ to the dimensional facility $P_i$ is related to its shape. In particular, we can consider the following cases:

\begin{enumerate}[$\bullet$]

\item Utility induced by the Minkowski functional (see \cite{rockafellarConvex}): assume that the dimensional facility $P_i$ is closed, convex, with non empty interior, then
$P_i$ induces a gauge $\gamma_{P_i}: \mathbb{R}^2\to \mathbb{R}$ defined by the Minkowski functional
$$ \gamma_{P_i}(q) =  \inf\left\{ \lambda > 0 : q\in \lambda P_i^{\underline{0}}\right\}, $$
where $\underline{0}=(0,0)\in \mathbb{R}^2$ and $\lambda P_i^{\underline{0}}$ denotes the resulting set from applying the homothecy of center $\underline{0}$ and ratio $\lambda$ to the set  $P_i^{\underline{0}}$. Observe that $\gamma_{P_i}(q)=1$ if $q\in\partial P_i^{\underline{0}}$ and that $\gamma_{P_i}(q)<1$ if $q\in\text{int}(P_i^{\underline{0}})$. Hence, a way to measure how far is the customer point $q\in \Omega$ from the dimensional facility $P_i^{\underline{0}}$ is using the continuous functional
$$ \tilde{\gamma}_{P_i}(q) = \begin{cases}  \gamma_{P_i}(q)-1, & \hbox{if $q\notin P_i^{\underline{0}}$,} \\  0, & \hbox{if $q\in P_i^{\underline{0}}$,} \end{cases}  $$
where only the points in the set $P_i^{\underline{0}}$ have assigned the value $0$.
Taking into account the above discussion, a natural way to define the utility in this context is
$$ u(q,P_i^{q_i}) = \begin{cases}  f_i(\gamma_{P_i}(q-q_i)-1), & \hbox{if $q\notin P_i^{q_i}$,} \\  f_i(0), & \hbox{if $q\in P_i^{q_i}$,} \end{cases}  $$
where $f_i:\mathbb{R} \to \mathbb{R}$ is a continuous function. Note that the utility $u(q,P_i^{q_i})$ depends on the root point $p_i$ of the dimensional facility $P_i$ as well as of its shape.

\item Conservative planner: this is the case in which the utility $u$ obtained from a customer point $q\in \Omega$ with respect to the dimensional facility $P_i^{q_i}$ is chosen as the maximum distance between $q$ and $P_i^{q_i}$ (see \cite{Brazil2014}), i.e., $u(q,P_i^{q_i}) = \max_{\tilde{q}\in P_i^{q_i}} \gamma_i(q-\tilde{q}),$
being $\gamma_i$ a norm. Or more generally,
$ u(q,P_i^{q_i}) = f_i\left(\max_{\tilde{q}\in P_i^{q_i}} \gamma_i(q-\tilde{q})\right), $
where  $f_i:\mathbb{R} \to \mathbb{R}$ is a continuous function. In the particular case in which $P_i$ is a polygon, we observe that the utility can be obtained as
$ u(q,P_i^{q_i}) = f_i\left(\max_{j=1,...,n_i} \gamma_i(q-[q_i+v_j^i]) \right),$
where $\{v_1^i,...,v_{n_i}^i\}$ are the vertices of $P_i^{\underline{0}}$. This last observation is interesting from a computational point of view. 

\end{enumerate}

\end{enumerate}


Given a suitable solution $Q = (q_1,...,q_\rho)\in \Gamma$, we introduce the notation $\Omega(Q)=\Omega \setminus \{\text{int}(P_1^{q_1})\cup ... \cup \text{int}(P_{\rho}^{q_\rho})\}$ to indicate the region of $\Omega$ to be partitioned as a function of the location of the dimensional facilities. In addition, we denote by $\mathcal{A}_{\rho}(Q)$ the set of all partitions, up to $D$-negligible sets, in $\rho$ sub-regions of the region
$\Omega(Q)$ and by $A(Q) = (A_1(Q),..., A_{\rho}(Q))$ an element of $\mathcal{A}_{\rho}(Q)$.

In the spirit of a \emph{social planner}, we are interested in  finding a partition $A(Q)=(A_1(Q),...,A_{\rho}(Q))$  of the customers in $\Omega(Q)$ solving the problem:
$$ \displaystyle  \min_{A(Q)\in \mathcal{A}_{\rho}(Q)}\sum_{i=1}^{\rho}   \left\lbrace \int_{{A}_i(Q)}  [ a_i + u(q, P_i^{q_i}) ] D(q) dq    \right\rbrace,    \eqno{\text{LL$(Q)$}}   $$
where $a_i>0$ is the cost incurred by each customer to access dimensional facility $P_i$ per unit demand and the second term in each integral
$$U_i(A_i(Q))=  \int_{{A}_i(Q)} u(q, P_i^{q_i}) D(q) dq$$
is the  distribution cost in the service region ${A}_i(Q)$, for each $i\in \{1,...,\rho\}$.



In a second step, the planner proposes the best location of the $\rho$ facilities in such a way that some additional costs  are minimized, knowing that, given a suitable solution $Q$, the best partition of the customers is given by solving the lower level problem LL$(Q)$. These additional costs are: 1) the installation cost of each facility; 2) a cost due to the waiting time to be served by each facility; 3) a cost induced by the demand that is lost. In the following we describe in detail these costs.

\begin{enumerate}[1)]
\item \textbf{Installation cost}: suppose that in  $\Omega$, besides of a demand density $D$, there exists another absolutely continuous measure $B\in \mathcal{L}^2( \Omega)$ to model the base installation costs. We assume that $B: \Omega \to \mathbb{R}$  is a nonnegative function with finite integral
$\int_\Omega B(q) dq < \infty$.
For a suitable solution $Q=(q_1,...,q_{\rho})\in \Gamma$, the installation cost of the dimensional facility $P_i$ is modeled by the non-decreasing continuous function $I_i: \omega^I_i\in \mathbb{R}  \to I_i(\omega^I_i)\in [0,+\infty)\subseteq \mathbb{R}$, being   $\omega^I_i= \int_{P_i^{q_i}}  B(q) dq$, for each $i\in\{1,...,\rho\}$. There are many realistic installation costs that fit within this framework: standard set up cost fits by taking $I_i(\omega^I_i)=F_i\in\mathbb{R}$ for all $\omega^I_i\in \mathbb{R}$; square meter cost is obtained assuming that $B$ is the density of the square meter cost in $\Omega$ and that $F_i\in\mathbb{R}$ is the fixed cost of building the dimensional facility $P_i$, then the  installation cost of $P_i$ is $I_i(\omega^I_i)=F_i+\int_{P_i^{q_i}} B(q) dq $ for all $\omega^I_i\in \mathbb{R}$, $i\in\{1,...,\rho\}$; square meter cost with economy of scale also fits taking  $I_i(\omega^I_i)=F_i+\tilde{I}_i(\int_{P_i^{q_i}} B(q) dq )$, being $\tilde{I}_i: \omega^I_i\in \mathbb{R}  \to \tilde{I}_i(\omega^I_i)\in [0,+\infty)\subseteq \mathbb{R}$ a non-decreasing, continuous and concave function, for all $\omega^I_i\in \mathbb{R}$ and $i\in\{1,...,\rho\}$.

\item \textbf{Congestion cost}: if $A(Q)=(A_1(Q),...,A_{\rho}(Q))$ is a partition of the customers in $\Omega(Q)$ for a suitable solution $Q\in \Gamma$, we consider the congestion cost $C_i: \omega^C_i\in [0,1] \to C_i(\omega_i^C)\in [0,+\infty)\subseteq\mathbb{R}$ for facility $P_i$, where $\omega^C_i= \int_{{A}_i(Q)}  D(q) dq$ and $C_i$ is non-decreasing and continuous, for any $i\in \{1,...,\rho\}$. Congestion cost is the most relevant of the above mentioned additional costs, since as we will see, it induces in our problem a hierarchical structure of bilevel optimization.

\item \textbf{Lost demand cost}: a lost demand cost is computed over the lost demand in $\Omega\setminus \Omega (Q)=\{\text{int}(P_1^{q_1})\cup...\cup \text{int}(P_{\rho}^{q_\rho})\}$. Lost demand cost is given by $L: \omega^L\in [0,1] \to L(\omega^L)\in [0,+\infty)\subseteq\mathbb{R}$, being $L$ a non-decreasing and continuous function, and where  $\omega^L=\int_{P_1^{q_1}\cup... \cup P_{\rho}^{q_{\rho}}} D(q)dq = \sum_{i=1}^{\rho} \int_{P_i^{q_i}} D(q) dq$. We are assuming that demand in the region $P_i^{q_i}\subseteq \Omega$ is incompatible with installation of $P_i$ within that region, for any $i\in\{1,...,\rho\}$, and therefore, lost demand has to be accounted for. This assumption can be dropped taking $L(\omega^L)=0$ for all $\omega^L\in [0,1]$.

\end{enumerate}

The costs above induce the following constrained optimization problem. The optimal suitable solution of the dimensional facilitites $Q=(q_1,...,q_{\rho})\in \Gamma$ can be obtained solving the following  bilevel problem:
$$ \begin{array}{ll} \displaystyle  \min  &    \mathcal{F}(Q) \\
s.t. &   \widehat{A}(Q)\in \arg\displaystyle  \min_{A(Q)\in \mathcal{A}_{\rho}(Q)}\sum_{i=1}^{\rho}   \left\lbrace \int_{{A}_i(Q)}  [ a_i + u(q, P_i^{q_i}) ] D(q) dq    \right\rbrace,   \\
     &   Q \in \Gamma,  \end{array}  \eqno{\text{BL}} $$
being
$$ \mathcal{F} (Q) := \displaystyle   \sum_{i=1}^{\rho}  \left[ I_i \biggl( \int_{P_i^{_i}} B(q) dq \biggr)  +  C_i \biggl( \int_{\widehat{A}_i (Q)}  D(q) dq \biggr)\right]
       + L\biggl( \int_{P_1^{q_1}\cup ... \cup P_{\rho}^{q_{\rho}}} D(q) dq\biggr).$$
Observe that for a given suitable solution $Q\in \Gamma$, the partition $\widehat{A}(Q)$ of $\Omega(Q)$ is given by a solution of problem LL$(Q)$. The solution of the location-allocation problem will be the pair $(Q^*, \widehat{A}(Q^*))$ where $Q^*$ solves problem BL.
Let us remark that if $Q^*=(q^*_1,...,q^*_{\rho})$ is an optimal suitable solution of the bilevel problem BL, then for any $i\in \{1,...,{\rho}\}$, if the dimensional facility $P_i^{q_i^*}$ is part of the optimal suitable solution then it  is uniquely  determined by the location $q_i^*$ of its root point $p_i$, since we are assuming that its shape is fixed.


\subsection{Resolution via optimal transport mass}


Consider problem LL$(Q)$ for a given suitable solution $Q\in\Gamma$. We point out that for dimensional facilities, we can not directly apply the optimal transport theory as done in \cite{CarMallozzi2018,MaPa,optPartition}, because the characterization  of the optimal partition holds when the measure $\nu$ has a discrete support. However, we can prove the existence of solution for problem LL$(Q)$ by identifying each dimensional facility with its root point, giving to the measure a discrete support, as the proof of the following theorem shows. Thus, building upon the results that appear in the mentioned works, we can obtain a result similar to the one given in those papers but applicable in this more general framework.

\begin{Thm}\label{bestReply}
Let $Q=(q_1,...,q_{\rho})\in \Gamma$. Suppose that the set
\begin{equation}\label{Dnegligible}
\{q\in \Omega(Q): a_i + u(q,P_i^{q_i}) = a_j + u(q,P_j^{q_j})\}
\end{equation}
is $D$-negligible, for all $i,j\in \{1,...,\rho\}$ with $i\not= j$.
Then problem LL$(Q)$ admits a unique solution $A(Q)=(A_1(Q),...,A_{\rho}(Q))$ that verifies
\begin{equation}\label{optPart}
A_i(Q) = \{q\in \Omega(Q): a_i + u(q,P_i^{q_i}) < a_j + u(q,P_j^{q_j}), \forall j\in \{1,...,\rho\}, j\not=i\}
\end{equation}
for each $i\in\{1,...,\rho\}$, where the equalities are intended up to $D$-negligible sets.
\end{Thm}
\begin{proof} To prove the existence of solution for problem LL$(Q)$, we rewrite it as a Monge optimal transport problem (see Section 2.1 in \cite{optPartition}). In the proof, we use the absolutely continuous probability measure $\tilde{\mu}(q) = \tilde{D}(q) dq$ being $\tilde{D}(q) = \dfrac{1}{\int_{\Omega(Q)} D(q)dq} D(q)$. Indeed, we prove the existence of solution for the auxiliary problem
\begin{equation}
\inf_{A(Q)\in \mathcal{A}_{\rho}(Q)} \sum_{i=1}^{\rho}   \left\lbrace \int_{{A}_i(Q)}  [ a_i + u(q, P_i^{q_i}) ] \tilde{D}(q) dq    \right\rbrace , \label{LLtilde}
\end{equation}
which implies the existence of solution for problem LL$(Q)$.

Let $S$ be the unit simplex in $\mathbb{R}^{\rho}$ defined by $S=\left\lbrace \omega =(\omega_1,...,\omega_{\rho})\in \mathbb{R}^{\rho} : \omega_i\geq 0, \sum_{i=1}^{\rho} \omega_i = 1\right\rbrace$. Then, we can rewrite  problem (\ref{LLtilde}) in the following form:
\begin{equation}
\inf_{\omega\in S}  \left(  \inf_{A(Q)\in \mathcal{A}_{\rho}(Q)} \left\lbrace \sum_{i=1}^{\rho}   \left[ \int_{{A}_i(Q)} u_i(q, q_i) \tilde{D}(q)dq \right] : \int_{{A}_i(Q)} \tilde{D}(q)dq = \omega_i \right\rbrace + \sum_{i=1}^{\rho} a_i \omega_i \right). \label{rewrite1}
\end{equation}

Let $\tilde{q}_1,...,\tilde{q}_{\rho}$ be any $\rho$ points in $\Omega (Q)$ such that $\tilde{q}_i\not= \tilde{q}_j$, for all $i,j\in \{1,...,\rho\}$ with $i\not=j$. By Tietze's extension theorem, there exists a continuous function $c: \Omega(Q)\times \Omega(Q) \to [0,+\infty]$ such that $c(q,\tilde{q}_i) = u_i(q,q_i)$, for any  $q\in \Omega(Q)$ and $i\in\{1,...,\rho\}$. Given $\omega = (\omega_1,...,\omega_{\rho})\in S$, consider the Monge optimal transport problem
\begin{equation}
\inf_{T_\sharp \tilde{\mu} = \nu(\omega)} \int_{\Omega(Q)} c(q,T(q))d\tilde{\mu}(q) \label{OpTransport}
\end{equation}
being $\nu(\omega) = \sum_{i=1}^{\rho}  \omega_i \delta_{\tilde{q}_i}$.

By Theorem 2.1 in \cite{optPartition} there exists a solution for problem (\ref{OpTransport}) and it is equivalent to its corresponding Kantorovich relaxed Monge's formulation:
\begin{equation}
\inf_{T_\sharp \tilde{\mu} = \nu(\omega)} \int_{\Omega(Q)} c(q,T(q))d\tilde{\mu}(q) = \mathcal{W}_c(\tilde{\mu},\nu(\omega)). \label{relaxation}
\end{equation}

By Remark 1 in \cite{MaPa}, in the problem (\ref{OpTransport}) any transport map $T$ is associated to a partition $(A_i)_{i=1}^{\rho}$ of $\Omega(Q)$ in such a way that
$$T(q)=\sum_{i=1}^{\rho} \tilde{q}_i \bm{1}_{A_i}(q) \quad \text{and} \quad \tilde{\mu}(A_i) = \omega_i.$$
Conversely any partition $(A_i)_{i=1}^{\rho}$ of $\Omega(Q)$ satisfaying $\tilde{\mu}(A_i)=\omega_i$ corresponds to a transport map of the form above. Then, we have that
\begin{eqnarray}
&& \inf_{T_\sharp \tilde{\mu} = \nu(\omega)} \int_{\Omega(Q)} c(q,T(q))d\tilde{\mu}(q)  = \inf_{A(Q)\in \mathcal{A}_{\rho}(Q)} \left\lbrace \int_{\Omega(Q)} c(q,\sum_{i=1}^{\rho} \tilde{q}_i \bm{1}_{A_i}(q))d\tilde{\mu}(q) :  \tilde{\mu}(A_i)=\omega_i \right\rbrace \nonumber \\
&& = \inf_{A(Q)\in \mathcal{A}_{\rho}(Q)} \left\lbrace \sum_{i=1}^{\rho}   \left[ \int_{{A}_i(Q)} u_i(q, q_i) \tilde{D}(q)dq \right] : \int_{{A}_i(Q)} \tilde{D}(q)dq = \omega_i \right\rbrace . \label{remarkResult}
\end{eqnarray}

Using equalities (\ref{LLtilde}) $ = $ (\ref{rewrite1}), (\ref{relaxation}) and (\ref{remarkResult}), we rewrite problem (\ref{LLtilde}) as:
\begin{equation*}
 \inf_{A(Q)\in \mathcal{A}_{\rho}(Q)} \sum_{i=1}^{\rho}   \left\lbrace \int_{{A}_i(Q)}  [ a_i + u(q, P_i^{q_i}) ] \tilde{D}(q) dq    \right\rbrace  = \inf_{\omega\in S}  \left\lbrace  \mathcal{W}_c(\tilde{\mu},\nu(\omega)) + \sum_{i=1}^{\rho} a_i \omega_i \right\rbrace. \label{rewrite2}
\end{equation*}
The function $\mathcal{W}_c(\tilde{\mu},\nu(\cdot)) : S \to \mathbb{R}$ is continuous since $\mathcal{W}_c$ is the Wasserstein distance on the set $\mathcal{P}(\Omega(Q))$ of Borel probability measures on $\Omega(Q)$. As in addition $S$ is compact, there exists a minimizer for problem (\ref{LLtilde}).

The form and the uniqueness of the solution for problem LL$(Q)$ is obtained adapting the proofs of Lemma 2 and Theorem 2 in \cite{MaPa}, respectively. \qed
\end{proof}

Theorem \ref{bestReply} ensures problem LL$(Q)$ is feasible, moreover, explicitly gives the unique  solution, up to $D$-negligible sets, of the problem. Note that the unique solution of problem LL$(Q)$ given in Theorem \ref{bestReply} represents the natural choice of each customer point in $\Omega(Q)$ given a prescribed utility, i.e., each customer point decides to be served by the dimensional facility that charges him the lowest cost. So, the form of the solution (\ref{optPart}) provides a realistic modeling of the customers' behaviour.

For each particular case of utility and shape of the facilities, the condition that
(\ref{Dnegligible})
is $D$-negligible for all $i,j\in \{1,...,\rho\}$ with $i\not= j$, has to be guaranteed to  ensure that Theorem \ref{bestReply} is applicable. For example, for the conservative planner case and polygonal facilities, the condition is guaranteed for all $Q\in \Gamma$ whenever  $a_i\not= a_j$ for all $i,j\in \{1,...,\rho\}$ with $i\not=j$. This is not a strong assumption since the case $a_i=a_j$ can be tackle by slightly perturbing the values: $a_i+\varepsilon=a_j$ or $a_i=a_j+\varepsilon$ with $\varepsilon > 0$ small enough. Onwards, we assume that the hypothesis of Theorem \ref{bestReply} is satisfied for all $Q\in \Gamma$.

As the solution of problem LL$(Q)$ is unique for all $Q\in \Gamma$, we can define the \emph{best reply function} $\widehat{A}: Q\in \Gamma\to  \widehat{A}(Q)\in \mathcal{A}_{\rho}(Q)$, that  maps to a given suitable solution, the optimal partition of the customers given in (\ref{optPart}). In the same way, the function $\widehat{A}_i$ is the $i$-th projection of the function $\widehat{A}$, for each $i\in \{1,...,\rho\}$.

Taking into account the above, we can prove the existence of solution for problem BL.

\begin{Lem}\label{omegaClosed}
For any $i\in\{1,...,\rho\}$, the set $\Omega_i$ is closed. In addition, the set $\Omega_1\times ... \times \Omega_{\rho}$ is also closed.
\end{Lem}
\begin{proof}
To prove this statement, it is enough to show that $\mathbb{R}^2 \setminus \Omega_i =  \{q_i\in\mathbb{R}^2 : P_i^{q_i}\nsubseteq \Omega\}$ is open. Let $q_i\in \mathbb{R}^2 \setminus \Omega_i$. Then, there is a point $q\in P_i^{q_i}$ such that $q\not \in \Omega$. As $\mathbb{R}^2$ is regular with the usual topology, there exist two open sets $Z_1$ and $Z_2$ such that $q\in Z_1$, $\Omega\subseteq Z_2$ and $Z_1\cap Z_2 = \emptyset$. Let $\varepsilon> 0$ such that $\mathcal{B}_{e}(q,\varepsilon)\subseteq Z_1$, being $\mathcal{B}_{e}(q,\varepsilon)$  the open Euclidean ball centered at $q$ with radius $\varepsilon$. Now, note that every point $\tilde{q}_i\in \mathcal{B}_{e}(q_i,\varepsilon)$ verifies: $q+\overrightarrow{q_i\tilde{q}_i}\in P_i^{\tilde{q}_i}$; $q+\overrightarrow{q_i\tilde{q}_i}\notin \Omega$, since $q+\overrightarrow{q_i\tilde{q}_i}\in \mathcal{B}_{e}(q,\varepsilon)$ and $\mathcal{B}_{e}(q,\varepsilon)\cap \Omega = \emptyset$. Thus, $\tilde{q}_i\in \mathbb{R}^2\setminus \Omega_i$ for all $\tilde{q}_i\in \mathcal{B}_{e}(q_i,\varepsilon)$, and this means that $\mathbb{R}^2\setminus \Omega_i$ is open.

As $\Omega_i$ is closed for each $i\in\{1,...,\rho\}$,  the set $\Omega_1\times ... \times \Omega_{\rho}$ is closed in $\mathbb{R}^{2\rho}$ with the usual topology because it is a product of closed sets. \qed
\end{proof}

\begin{Lem}\label{lemCompact}
The set $\Gamma $ is compact.
\end{Lem}
\begin{proof}
Actually, we have to prove that the set $\Gamma\subseteq \mathbb{R}^{2\rho}$ is closed, since the fact that $\Gamma$ is bounded is clear. To do this, we prove that $ \mathbb{R}^{2\rho}\setminus \Gamma$ is open. Note that $\mathbb{R}^{2\rho} \setminus \Gamma = \Gamma^{\complement}_1 \cup \Gamma^{\complement}_2$ being $\Gamma^{\complement}_1= \{(q_1,...,q_{\rho})\in \mathbb{R}^{2 \rho} : P_i^{q_i}\nsubseteq \Omega \text{ for some } i\in\{1,...,\rho\}\}$ and $\Gamma^{\complement}_2 = \{(q_1,...,q_{\rho})\in \mathbb{R}^{2 \rho} : \text{int}(P_i^{q_i})\cap \text{int}(P_j^{q_j})\not= \emptyset  \text{ for some } i,j\in\{1,...,\rho\}\text{ with } i\not= j\}$. So, if the sets $\Gamma^{\complement}_1$ and $\Gamma^{\complement}_2$ are open, then the set $\mathbb{R}^{2\rho} \setminus \Gamma$ will be open.

Observe that $\Gamma^{\complement}_1 = \mathbb{R}^{2\rho} \setminus \{\Omega_1 \times ... \times \Omega_{\rho}\}$. Thus, as the set $\Omega_1 \times ... \times \Omega_{\rho}$ is closed by Lemma \ref{omegaClosed}, the set $\Gamma^{\complement}_1$ is open. Now consider a point $(q_1,...,q_{\rho})\in \Gamma^{\complement}_2$. Then there exist $i,j\in\{1,...,\rho\}$ with $i\not= j$ such that $\text{int}(P_i^{q_i})\cap \text{int}(P_j^{q_j}) \not= \emptyset$. Let $q\in \text{int}(P_i^{q_i})\cap \text{int}(P_j^{q_j})$ and let $\mathcal{B}_{\infty}(q,\varepsilon)$  be any open ball  centered at $q$ with radius $\varepsilon> 0$, with respect to the maximum metric $\ell_{\infty}$, such that $\mathcal{B}_{\infty}(q,\varepsilon)\subseteq \text{int}(P_i^{q_i})\cap \text{int}(P_j^{q_j})$. Then, it can be proven that $\mathcal{B}_{\infty}((q_1,...,q_{\rho}),\varepsilon / 2) \subseteq \Gamma^{\complement}_2$, which implies that the set $\Gamma^{\complement}_2$ is open. To see the inclusion above, note that the  ball $\mathcal{B}_{\infty}((q_i,q_j),\varepsilon / 2)$ of $\mathbb{R}^4$ is contained in $\mathcal{B}_{\infty}((q_1,...,q_{\rho}),\varepsilon / 2)$, and that $\text{int}(P_i^{\tilde{q}_i})\cap \text{int}(P_j^{\tilde{q}_j}) \not= \emptyset$ for all $(\tilde{q}_i,\tilde{q}_j)\in \mathcal{B}_{\infty}((q_i,q_j),\varepsilon / 2)$.\qed
\end{proof}

\begin{Thm}\label{existenceThm}
There exists an optimal solution for problem BL.
\end{Thm}
\begin{proof}
Using the function $\widehat{A}_i$ defined as above for each $i\in \{1,...,\rho\}$ and Lemma \ref{lemCompact},  problem BL consists in minimising a continuous function $\mathcal{F}$ on a compact set $\Gamma$.
We will get the result using Weierstrass theorem.

To prove that $\mathcal{F}$ is continuous on $\Gamma$, it is enough to prove that each one of its summands is continuous on $\Gamma$. We give full details of the proof for the functions
$\mathscr{C}_i (Q)= C_i \biggl( \int_{\widehat{A}_i (Q)}  D(q) dq \biggr)$
and we only outline the proof for the remaining functions
$\mathscr{I}_i (Q)=I_i \biggl( \int_{P_i^{q_i}} B(q) dq \biggr)$  and
$\mathscr{L} (Q)=L\biggl( \int_{P_1^{q_1}\cup ... \cup P_{\rho}^{q_{\rho}}} D(q) dq\biggr)$,
for any $i\in \{1,...,\rho\}$, since the proofs are similar.

Take $i\in\{1,...,\rho\}$. Let $\mu_i: \widehat{A}(Q)\in \widehat{A}(\Gamma) \to \int_{\widehat{A}_i(Q)}  D(q) dq\in [0,1]\subseteq \mathbb{R}$, i.e., $\mu_i(\widehat{A}(Q))$ is the measure, with respect to the $D$ density, of the $i$-th component of $\widehat{A}(Q)$. Note that $\mathscr{C}_i (Q) = C_i(\mu_i(\widehat{A}(Q)))$ for all $Q\in \Gamma$. So, as $C_i$ is continuous, if we prove that $\widehat{A}$ and $\mu_i$ are continuous then $\mathscr{C}_i$ will be  continuous.

Consider the application between topological spaces $\widehat{A}:\Gamma\to  \widehat{A}(\Gamma)$, where $\Gamma$ is endowed with the relative topology of $\mathbb{R}^{2\rho}$ and  $\widehat{A}(\Gamma)$ with the final topology. As $\widehat{A}(\Gamma)$ is endowed with the final topology, $\widehat{A}$ is continuous as application between topological spaces. Moreover, $\widehat{A}$ is  a homeomorphism. Indeed, observe that $\widehat{A}(Q)$ is different for each $Q\in \Gamma$, since $\widehat{A}$ partitions a different set $\Omega(Q)$ for each $Q\in \Gamma$. Then, $\widehat{A}$ is injective and also bijective, since $\widehat{A}$ is clearly surjective. Thus, since the image space, $\widehat{A}(\Gamma)$, is endowed with the final topology, $\widehat{A}$ is a homeomorphism.


To prove that $\mu_i$ is continuous we have to show that $\mu_i^{-1}(Z)$ is open in $\widehat{A}(\Gamma)$ for any open set $Z$ in $\mathbb{R}$, where $\widehat{A}(\Gamma)$ is endowed with the final topology indicated above. Since the open Euclidean balls constitute a base of the usual topology, it is enough to consider open Euclidean balls, i.e., intervals $Z=(\alpha,\beta)$ with $\alpha,\beta\in \mathbb{R}$ and $\alpha < \beta$.

For any $Z=(\alpha,\beta)$ as above, we have that
$$\mu_i^{-1} (Z) = \left\lbrace \widehat{A}(Q) \in \widehat{A}(\Gamma) : \int_{\widehat{A}_i(Q)} D(q)dq\in (\alpha,\beta)\right\rbrace.$$
Let $\widehat{A}(\tilde{Q})\in \mu_i^{-1} (Z)$, where $\tilde{Q}=(\tilde{q}_1,...,\tilde{q}_{\rho})\in \Gamma$. Then, $\int_{\widehat{A}_i(\tilde{Q})} D(q)dq=\varsigma \in (\alpha,\beta)$. Next, we will prove that there exists $\epsilon>0$ such that $\widehat{A}(\tilde{Q}) \in \widehat{A}(\mathcal{B}_{\infty}(\tilde{Q}, \epsilon)\cap \Gamma) \subseteq \mu_i^{-1} (Z)$. That result implies that $\mu_i^{-1} (Z)$ is open, which will complete the proof. Note that $\mathcal{B}_{\infty}(\tilde{Q}, \epsilon)\cap \Gamma$ is the relative open ball $\mathcal{B}_{\infty}(\tilde{Q}, \epsilon)$ of $\mathbb{R}^{2\rho}$ in $\Gamma$, so it is open in $\Gamma$ endowed with the relative topology of $\mathbb{R}^{2\rho}$. Hence,  $\widehat{A}(\mathcal{B}_{\infty}(\tilde{Q}, \epsilon)\cap \Gamma)$ is also open in $\widehat{A}(\Gamma)$ endowed with the final topology mentioned above, because of $\widehat{A}$ is a homeomorphism. Therefore, $\widehat{A}(\mathcal{B}_{\infty}(\tilde{Q}, \epsilon)\cap \Gamma)$  is an open neighbourhood of $\widehat{A}(\tilde{Q})$ contained in $\mu_i^{-1} (Z)$, which means that   $\mu_i^{-1} (Z)$ is open.

\noindent \textbf{Claim} \,\emph{There exists $\epsilon>0$ such that $\widehat{A}(\tilde{Q}) \in \widehat{A}(\mathcal{B}_{\infty}(\tilde{Q}, \epsilon)\cap \Gamma) \subseteq \mu_i^{-1} (Z)$.}

\noindent \emph{Proof of the Claim} Let $\varepsilon > 0$ be small enough. For each $\varepsilon_n = \varepsilon/n$ with $n\in \mathbb{N}$, we define the following sets: $\Omega^-(\tilde{Q},\varepsilon_n) = \{q\in \Omega : q\notin \text{int}(P_1^{\breve{q}_1})\cup ... \cup \text{int}(P_{\rho}^{\breve{q}_{\rho}})  \text{ for all } \breve{Q}=(\breve{q}_1,...,\breve{q}_{\rho})\in \mathcal{B}_{\infty}(\tilde{Q},\varepsilon_n)\}$ and $\Omega^+(\tilde{Q},\varepsilon_n) = \{q\in \Omega : q\in \Omega\setminus\{\text{int}(P_1^{\breve{q}_1})\cup ... \cup \text{int}(P_{\rho}^{\breve{q}_{\rho}}) \}  \text{ for some } \breve{Q}=(\breve{q}_1,...,\breve{q}_{\rho})\in \mathcal{B}_{\infty}(\tilde{Q},\varepsilon_n)\}.$
It is not difficult to see that the sets $\Omega^-(\tilde{Q},\varepsilon_n)$ and $\Omega^+(\tilde{Q},\varepsilon_n)$ are measurable with respect to the Lebesgue measure $m$. Now, consider the sets
$\widehat{A}_i^-(\tilde{Q},\varepsilon_n) = \{q\in \Omega^-(\tilde{Q},\varepsilon_n) :  a_i + u(q,P_i^{\tilde{q}_i}) + 3\xi_n < a_j + u(u,P_j^{\tilde{q}_j}) \text{ for all } j\in\{1,...,\rho\} \text{ with } j\not= i\}$
and
$\widehat{A}_i^+(\tilde{Q},\varepsilon_n) = \{q\in \Omega^+(\tilde{Q},\varepsilon_n) :  a_i + u(q,P_i^{\tilde{q}_i})  < a_j + u(q,P_j^{\tilde{q}_j}) + 3\xi_n \text{ for all } j\in\{1,...,\rho\} \text{ with } j\not= i\}$,
being
$\xi_n = \max \left\lbrace \left|u(q,P_j^{\tilde{q}_j})-u(q,P_j^{\breve{q}_j})\right| : q\in \Omega, j\in \{1,...,\rho\}, \breve{q}_j\in \text{cl}(\mathcal{B}_{\infty}(\tilde{q}_j,\varepsilon_n))\right\rbrace$,
which are also Lebesgue measurable sets.

Note that $ \widehat{A}_i^-(\tilde{Q},\varepsilon_n) \subseteq \widehat{A}_i^-(\tilde{Q},\varepsilon_{n+1}) \subseteq \widehat{A}_i(\tilde{Q}) $ since
$ \Omega^-(\tilde{Q},\varepsilon_n) \subseteq \Omega^-(\tilde{Q},\varepsilon_{n+1}) \subseteq \Omega(\tilde{Q})$. Analogously,
$ \widehat{A}_i(\tilde{Q})  \subseteq  \widehat{A}_i^+(\tilde{Q},\varepsilon_{n+1}) \subseteq \widehat{A}_i^+(\tilde{Q},\varepsilon_n) $
since
$ \Omega(\tilde{Q}) \subseteq \Omega^+(\tilde{Q},\varepsilon_{n+1}) \subseteq \Omega^+(\tilde{Q},\varepsilon_{n})$.
Indeed,
$$\bigcup_{n=1}^{\infty}  \widehat{A}_i^-(\tilde{Q},\varepsilon_n) = \widehat{A}_i(\tilde{Q}) = \bigcap_{n=1}^{\infty}  \widehat{A}_i^+(\tilde{Q},\varepsilon_n)$$
up to $m$-negligible sets. Thus, applying the continuity  properties of the Lebesgue measure, it follows that
\begin{small}
\begin{equation}\label{lebesgueCont}
\lim_{n\to \infty}m\left(\widehat{A}_i^-(\tilde{Q},\varepsilon_n)\right) = m\left(\bigcup_{n=1}^{\infty}  \widehat{A}_i^-(\tilde{Q},\varepsilon_n)\right) = m\left(\widehat{A}_i(\tilde{Q}) \right) = m\left(\bigcap_{n=1}^{\infty}  \widehat{A}_i^+(\tilde{Q},\varepsilon_n)\right)= \lim_{n\to \infty} m\left(\widehat{A}_i^+(\tilde{Q},\varepsilon_n)\right).
\end{equation}
\end{small}

Recall that as $D$ is an absolutely continuous measure, for every $\phi >0$ there exists $\eta>0$ such that $\int_{Z} D(q)dq < \phi$ for every Lebesgue measurable set $Z$ for which $m(Z)<\eta$. Let $\phi >0$ be such that $(\varsigma-\phi,\varsigma+\phi)\subseteq (\alpha,\beta)$. Due to (\ref{lebesgueCont}), there always exists $n_0\in \mathbb{N}$ such that $m\left(\widehat{A}_i(\tilde{Q}) \setminus \widehat{A}_i^-(\tilde{Q},\varepsilon_{n_0})\right) < \eta$ and $m\left(\widehat{A}_i^+(\tilde{Q},\varepsilon_{n_0})\setminus \widehat{A}_i(\tilde{Q}) \right) < \eta$
for all $\eta >0$. Therefore, we can find a $n_0\in \mathbb{N}$ for a $\eta >0$ which makes
$\int_{\widehat{A}_i(\tilde{Q}) \setminus \widehat{A}_i^-(\tilde{Q},\varepsilon_{n_0})} D(q)dq < \phi$ and $\int_{\widehat{A}_i^+(\tilde{Q},\varepsilon_{n_0})\setminus \widehat{A}_i(\tilde{Q})} D(q)dq < \phi$, or equivalently,
$\int_{\widehat{A}_i^-(\tilde{Q},\varepsilon_{n_0})} D(q)dq > \varsigma - \phi$ and $\int_{\widehat{A}_i^+(\tilde{Q},\varepsilon_{n_0})} D(q)dq < \varsigma + \phi$.
Take $n_0\in \mathbb{N}$ for which the above is true.

Now, let $\breve{Q}=(\breve{q}_1,...,\breve{q}_{\rho})\in \mathcal{B}_{\infty}(\tilde{Q},\varepsilon_{n_0})\cap \Gamma$. Then,
$\left| u(q,P_j^{\tilde{q}_j}) - u(q,P_j^{\breve{q}_j})\right| \leq \xi_{n_0}$
for all $q\in \Omega$ and $j\in \{1,...,\rho\}$. Therefore,
$\left| \left( u(q,P_i^{\tilde{q}_i})-u(q,P_j^{\tilde{q}_j}) \right) - \left( u(q,P_i^{\breve{q}_i})-u(q,P_j^{\breve{q}_j})\right)\right| \leq 2\xi_{n_0} < 3\xi_{n_0}$
for all $q\in \Omega$ and $j\in \{1,...,\rho\}$. Note that the above inequality together with the fact that
$\Omega^-(\tilde{Q},\varepsilon_{n_0}) \subseteq \Omega(\breve{Q}) \subseteq \Omega^+(\tilde{Q},\varepsilon_{n_0})$
imply that:
$$ \widehat{A}_i^-(\tilde{Q},\varepsilon_{n_0}) \subseteq \widehat{A}_i(\breve{Q}) \subseteq \widehat{A}_i^+(\tilde{Q},\varepsilon_{n_0}) .$$
Thus,
$$\int_{ \widehat{A}_i(\breve{Q})} D(q)dq \in (\varsigma - \phi, \varsigma + \phi)\subseteq (\alpha,\beta). $$
Hence, it is enough to take $\epsilon = \varepsilon_{n_0}$ to complete the proof of the Claim.

Reasoning in a similar way, it can be proven that $\mathscr{I}_1,...,\mathscr{I}_{\rho}$  and $\mathscr{L}$ are also continuous functions. To do this, for $\tilde{Q}\in \Gamma$ suitably taken, use the sets
$ P_i^-(\tilde{Q},\varepsilon_n) = \{q\in \Omega : q\in P_i^{\breve{q}_i} \text{ for all } \breve{Q}=(\breve{q}_1,...,\breve{q}_{\rho})\in \mathcal{B}_{\infty}(\tilde{Q},\varepsilon_n)\}$
and
$ P_i^+(\tilde{Q},\varepsilon_n) = \{q\in \Omega : q\in P_i^{\breve{q}_i} \text{ for some } \breve{Q}=(\breve{q}_1,...,\breve{q}_{\rho})\in \mathcal{B}_{\infty}(\tilde{Q},\varepsilon_n)\}$. \qed
\end{proof}

Theorem \ref{existenceThm} finally proves that problem BL is well-defined and gives sufficient conditions for the existence of optimal solutions.

\section{A convergent discrete approximation scheme}

The previous section states that problem BL is well-defined. However, in spite of being well-defined, optimizing problem BL is a very difficult task since it amounts to minimize with a best reply function over the partitions of $\Omega$ as a constraint defining the feasible domain. To overcome that inconvenience we propose a discrete approximation of problem BL. This approximation provides good solutions for the original problem. Since $\Omega$ is bounded by hypothesis, we can easily find a rectangle of $\mathbb{R}^2$ containing $\Omega$. Consider a grid $G$ over that rectangle, and thus over $\Omega$. Let $\bm{G}$ be the set of cells of the grid $G$. We denote by $(k,l)$ a cell of $\bm{G}$, where $k$ indexes the horizontal position of the cell in the grid and $l$ the vertical one. Now, consider the sets
$$\bm{\Omega} = \{(k,l)\in \bm{G} : \text{int}((k,l))\cap \Omega \not= \emptyset\}  $$
and
$$\OOmega = \bigcup_{(k,l)\in \bm{\Omega}} (k,l). $$
Clearly $\Omega \subseteq \OOmega$ and we want $\OOmega$ to be as similar to $\Omega$ as possible. Indeed, $\OOmega$ is the outer approximation of $\Omega$ given by the cells of the grid $G$ (see Fig. \ref{omegaKL}). The finer the grid, the better the approximation. Note that, for an element of the problem denoted by a letter, we use that letter in bold to represent the discrete counterpart of the element. Moreover, with the hollow fonts we represent the approximation of that element induced by its discrete counterpart, e.g., $\OOmega$ is the approximation of $\Omega$ induced by $\bm{\Omega}$. Onwards, we keep this meaning for the notation in bold and hollow fonts.

\begin{figure}[H]
\centering
\subfigure[$\Omega$]{\includegraphics[scale=0.4]{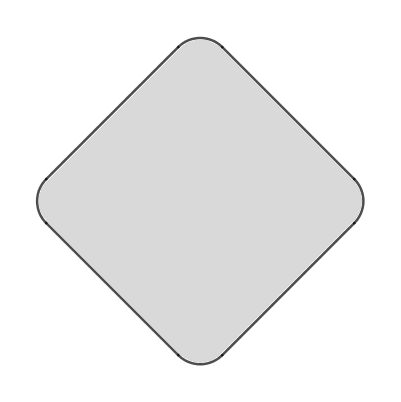}}
\hspace{1.5cm}
\subfigure[$\bm{\Omega}$]{\includegraphics[scale=0.4]{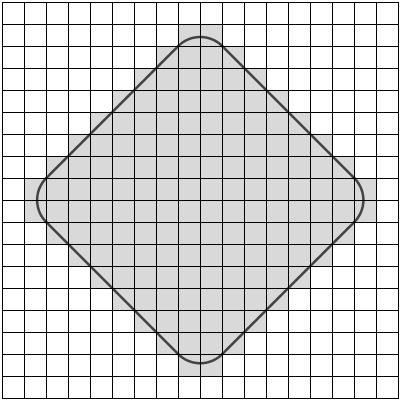}}
\hspace{1.5cm}
\subfigure[$\OOmega$]{\includegraphics[scale=0.4]{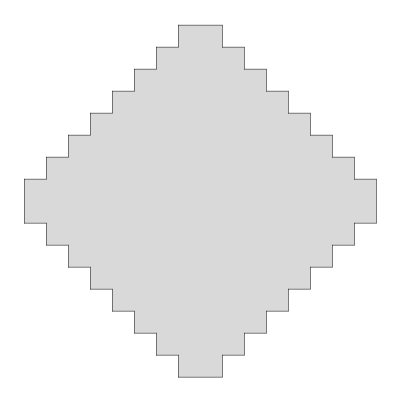}}
\caption{Sets $\bm{\Omega}$ and $\OOmega$ for an example of a Borel set $\Omega$ and a regular grid $G$} \label{omegaKL}
\end{figure}

Before to describe the discretization of problem BL, we introduce the following notation and define some elements involved in the discretization for each $(k,l)\in \bm{\Omega}$ and $i\in \{1,...,\rho\}$:
\begin{enumerate}[-]
\item $q_{(k,l)}$: is the center of the cell $(k,l)$ (if $\{(x^-,y^-),(x^+,y^-),(x^-,y^+),(x^+,y^+)\}\subseteq \mathbb{R}^2$ are the extreme points of the cell $(k,l)$, then the center of $(k,l)$ is $((x^-+x^+)/2,(y^-+y^+)/2)\in \mathbb{R}^2$).

\item $\bm{P}_i^{(k,l)}$: is the subset of cells of $\bm{\Omega}$ defined by
$$\bm{P}_i^{(k,l)} = \{(r,s)\in\bm{\Omega} : \text{int}((r,s))\cap \text{int}(P_i^{q_{(k,l)}})\not= \emptyset\}.$$


\item $\mathbb{P}_i^{(k,l)}$: is the set defined by
$$\mathbb{P}_i^{(k,l)} = \bigcup_{(r,s)\in \bm{P}_i^{(k,l)}} (r,s).$$
The set $\mathbb{P}_i^{(k,l)}$ is the approximation of the facility $P_i^{q_{(k,l)}}$ induced by the cells of $\bm{P}_i^{(k,l)}$ (the discretization scheme is the same that the one shown in Fig. \ref{omegaKL}). We refer to $\mathbb{P}_i^{(k,l)}$ as \emph{cell facility}. The finer the grid, the better the approximation.


\item $\bm{\Omega}_i$: is the subset of cells of $\bm{\Omega}$ defined by
$$\bm{\Omega}_i = \{(k,l)\in\bm{\Omega} : P_i^{q_{(k,l)}} \subseteq \Omega\}.$$

\item $\OOmega_i$: is the set defined by
$$\OOmega_i = \bigcup_{(k,l)\in \bm{\Omega}_i} (k,l).$$
The set $\OOmega_i$ is the approximation of $\Omega_i$ induced by $\bm{\Omega}_i$.
\end{enumerate}

The discretized version of problem BL (DBL) is to locate $\rho$ facilities $P_1,...,P_{\rho}$ in $\OOmega$  and to find their demand regions $A_1,...,A_{\rho}$ optimizing the costs as in the original continuous problem BL. To address this discretized problem we need to transform the original one making the following assumptions:
\begin{enumerate}[\text{Assumption} 1:]
\item The root points $p_1,...,p_{\rho}$ of the dimensional facilities $P_1,...,P_{\rho}$ can only be located at the centers of the cells in $\bm{\Omega}$. Then, a suitable solution of problem DBL is determined by a $\rho$-tuple $\bm{Q}=((k_1,l_1),...,(k_{\rho},l_{\rho}))\subseteq \bm{\Omega}^{\rho}$ where $(k_i,l_i)$ is the cell in whose center $q_{(k_i,l_i)}$ is located the root point of the dimensional facility $P_i$, for each $i\in \{1,...,\rho\}$. Therefore, in the discretized version of problem BL, dimensional facilities can only be placed in a finite number of locations.

\item We impose on the cell facilities some conditions induced by the corresponding ones applicable to the sets $P_1,...,P_{\rho}$ in problem BL.
\begin{enumerate}[\text{Assumption 2.}1:]

\item The interior of the cell facilities can not intersect between them. So, if we denote by $\bm{\Gamma}\subseteq \bm{\Omega}^{\rho}$ the set of suitable solutions of problem DBL, $\bm{Q}=((k_1,l_1),...,(k_{\rho},l_{\rho}))\in \bm{\Gamma}$ iff  $P_i^{q_{(k_i,l_i)}} \subseteq \Omega$ for all $i\in \{1,...,\rho\}$ and $\text{int}(\mathbb{P}_i^{(k_i,l_i)})\cap \text{int}(\mathbb{P}_j^{(k_j,l_j)}) = \emptyset$ for all $i,j\in \{1,...,\rho\}$ with $i\not= j$. Equivalently, using the sets defined above:
$$ \begin{array}{rl} \bm{\Gamma}=\{((k_1,l_1),...,(k_{\rho},l_{\rho}))\in \bm{\Omega}_1\times ... \times \bm{\Omega}_{\rho} : & \bm{P}_i^{(k_i,l_i)}\cap \bm{P}_j^{(k_j,l_j)} = \emptyset, \\ &\forall i,j\in \{1,...,\rho\},i\not= j\}. \end{array}$$
$$ .$$

\item Given a suitable solution $\bm{Q}=((k_1,l_1),...,(k_{\rho},l_{\rho}))\in \bm{\Gamma}$, instead of finding the optimal partition of $\OOmega\setminus \{\text{int}(P_1^{q_{(k_1,l_1)}})\cup ... \cup \text{int}(P_{\rho}^{q_{(k_{\rho},l_{\rho})}})\}$, we have to find the optimal partition of $\OOmega(\bm{Q}) = \OOmega\setminus \{\text{int}(\mathbb{P}_1^{(k_1,l_1)})\cup ... \cup \text{int}(\mathbb{P}_{\rho}^{(k_{\rho},l_{\rho})})\}$.
Note that $\OOmega(\bm{Q})$ is, up to $D$-negligible sets, the union of the cells of the set  $\bm{\Omega}(\bm{Q})$ defined as
$$\bm{\Omega}(\bm{Q})=\{(r,s)\in \bm{\Omega} : (r,s)\notin \bm{P}^{(k_i,l_i)}_i, \forall i\in\{1,...,\rho\}\}.$$

\item The installation and lost demand costs are computed now over the region occupied by the cell facilities.
\end{enumerate}


\item Given a suitable solution $\bm{Q}=((k_1,l_1),...,(k_{\rho},l_{\rho}))\in \bm{\Gamma}$, any partition $A(\bm{Q})=(A_1(\bm{Q}),...,A_{\rho}(\bm{Q}))$ of the set $\OOmega(\bm{Q})$ must satisfy that each region $A_i(\bm{Q})$ is the union of a finite number of cells of $\bm{\Omega}(\bm{Q})$, for all $i\in\{1,...,\rho\}$. We denote by $\bm{A}_i(\bm{Q})$ the subset of $\bm{\Omega}(\bm{Q})$ such that
$$A_i(\bm{Q}) = \bigcup_{(r,s)\in \bm{A}_i(\bm{Q})} (r,s),$$
for each $i\in\{1,...,\rho\}$. The partition $\bm{A}(\bm{Q})=(\bm{A}_1(\bm{Q}),...,\bm{A}_{\rho}(\bm{Q}))$ of $\bm{\Omega}(\bm{Q})$ assigns the demand cells in $\bm{\Omega}(\bm{Q})$ among the facilities. Note that, for this element of the problem, $A(\bm{Q})=\mathbb{A}(\bm{Q})$.

\item Suppose located the dimensional facilities $P_1,...,P_{\rho}$  according to $\bm{Q}=((k_1,l_1),...,(k_{\rho},l_{\rho}))\in \bm{\Gamma}$. The utility $u_G$ obtained from a point $q\in \OOmega(\bm{Q})$ with respect to the dimensional facility $P^{q_{(k_i,l_i)}}_i$ is now induced by the grid $G$ as:
\begin{equation*} \label{dcons2}
u_{G} (q, P_i^{q_{(k_i,l_i)}})= u (q_{(r,s)}, P_i^{q_{(k_i,l_i)}}),
\end{equation*}
being $q_{(r,s)}$ the center of the cell $(r,s)\in \bm{\Omega}(\bm{Q})$ to which the point $q$ belongs to, for each $i\in\{1,...,\rho\}$. Thus, in the discretized problem, all the points in a cell have the same utility, namely the utility of the center of that cell in the non-discretized problem. To ensure $u_G$ is well-defined, if $\{(x^-,y^-),(x^+,y^-),(x^-,y^+),(x^+,y^+)\}\subseteq \mathbb{R}^2$ are the extreme points of the cell $(r,s)$, in terms of membershipness, we consider $(r,s)$ as $[x^-,x^+)\times[y^-,y^+)\subseteq \mathbb{R}^2$ (this avoid that $q$ may belong to more than one cell). Note that if the grid $G$ is fine enough, $u_G(q, P_i^{q_{(k_i,l_i)}})$ gives a good approximation of $u(q,P_i^{q_{(k_i,l_i)}})$.


\item We assume that the cost functions $I_1,...,I_{\rho},C_1,...,C_{\rho},L$ are non-decreasing, continuous, with image on $[0,+\infty)$ and piecewise linear. We denote by $I_1^{\text{PL}},...,I_{\rho}^{\text{PL}},C_1^{\text{PL}},...,C_{\rho}^{\text{PL}},L^{\text{PL}}$ these cost functions in problem DBL to emphasize that they are piecewise linear. Note that the piecewise linearity assumption is not a big loss of generality. Indeed, taking a partition of the interval $[0,1]$ and evaluating the congestion cost function $C_i$ of problem BL at the points of the partition, we can build, by linear interpolation, a piecewise linear congestion cost function $C_i^{\text{PL}}$ that approximates $C_i$, for any $i\in\{1,...,\rho\}$. The finer the partition, the better the approximation. The same applies for $I_1,...,I_{\rho}$ and $L$.

\end{enumerate}

Fig. \ref{DULexample} shows, as an illustrative example, the discretized version of the Example 4.1 from \cite{optPartition} considering a regular grid $G$ over $\Omega$  with $25\times 25$ cells (note that, as $\Omega$ is the unit square, $\bm{G} = \bm{\Omega}$).

\begin{figure}[H]
\centering
\includegraphics[scale=0.15]{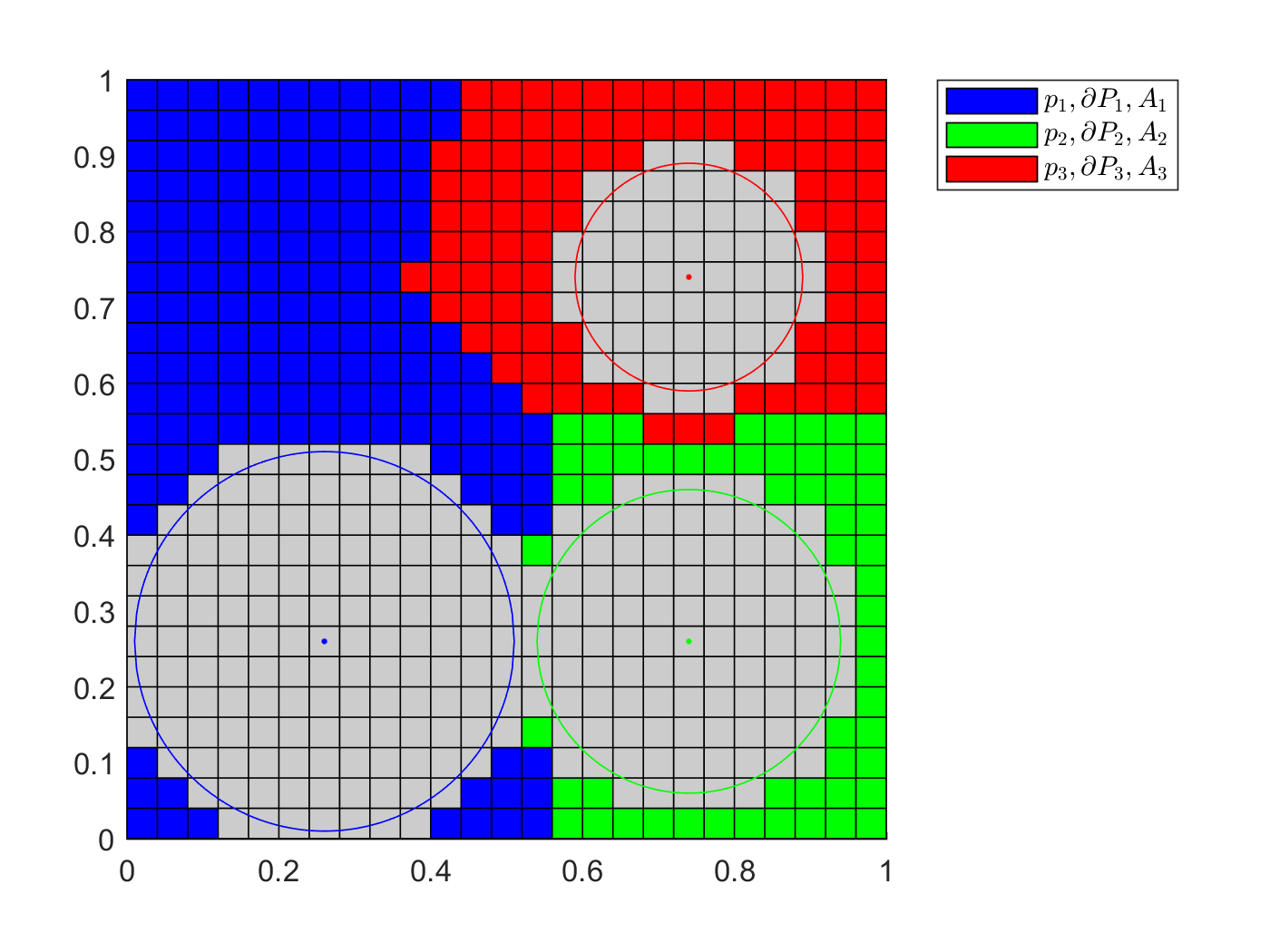}
\caption{Example 4.1 from \cite{optPartition} in the discrete scheme} \label{DULexample}
\end{figure}

Consider a suitable solution $\bm{Q}=((k_1,l_1),...,(k_{\rho},l_{\rho}))\in \bm{\Gamma}$ of  problem DBL and let $A(\bm{Q})=(A_1(\bm{Q}),...,A_{\rho}(\bm{Q}))$ be the optimal partition of the customers in $\OOmega(\bm{Q})$ under the assumptions above. Under those assumptions, for each $i\in\{1,...,\rho\}$, the access cost incurred by all customers assigned to the dimensional facility $P_i$ can be expressed as
\begin{equation*}
\int_{A_i(\bm{Q})} a_i D(q)dq =  \sum_{(r,s)\in \bm{A}_i(\bm{Q})} \int_{(r,s)} a_i D(q)dq= \sum_{(r,s)\in \bm{A}_i(\bm{Q})} a_i w^D_{rs},
\end{equation*}
where we are using the notation $ w^D_{rs} = \int_{(r,s)} D(q) dq$
for any $(r,s)\in \bm{\Omega}$.
Moreover, the distribution cost in the service region $A_i(\bm{Q})$ is
\begin{eqnarray*}
U_i(A_i(\bm{Q})) & = & \int_{A_i(\bm{Q})} u_G(q,P_i^{q_{(k_i,l_i)}}) D(q)dq \\
&=& \sum_{(r,s)\in \bm{A}_i(\bm{Q})} \left[ \int_{(r,s)}  u_G(q,P_i^{q_{(k_i,l_i)}}) D(q)dq \right]\\
&= & \sum_{(r,s)\in \bm{A}_i(\bm{Q})} \left[ \int_{(r,s)} u(q_{(r,s)},P_i^{q_{(k_i,l_i)}}) D(q)dq \right]\\
&=& \sum_{(r,s)\in \bm{A}_i(\bm{Q})}  w^D_{rs} u(q_{(r,s)},P_i^{q_{(k_i,l_i)}}),
\end{eqnarray*}
for each $i\in \{1,...,\rho\}$. Thus, the partition $A(\bm{Q})=(A_1(\bm{Q}),...,A_{\rho}(\bm{Q}))$ is given by the solution $\bm{A}(\bm{Q})=(\bm{A}_1(\bm{Q}),...,\bm{A}_{\rho}(\bm{Q}))$ of the discretized lower level problem
$$ \displaystyle  \min_{\bm{A}(\bm{Q})\in \bm{\mathcal{A}}_{\rho}(\bm{Q})} \sum_{i=1}^{\rho}   \sum_{(r,s)\in \bm{A}_i(\bm{Q})} [a_i w^D_{rs} + w^D_{rs}  u(q_{(r,s)},P_i^{q_{(k_i,l_i)}}) ], \eqno{\text{DLL$(\bm{Q})$}}   $$
being $\bm{\mathcal{A}}_{\rho}(\bm{Q})$ the set of all partitions in $\rho$ subsets (where the empty set is a valid subset) of the set $\bm{\Omega}(\bm{Q})$.

The assignment cost of a cell $(r,s)\in \bm{\Omega}(\bm{Q})$ to a dimensional facility $P_i$ in $\{P_1,...,P_{\rho}\}$ is $a_i w^D_{rs} + w^D_{rs}  u(q_{(r,s)},P^{q_{(k_i,l_i)}})$. Then, note that in problem DLL$(\bm{Q})$ we are minimizing the sum of the assignment costs of the cells in $\bm{\Omega}(\bm{Q})$.
Thus, the optimal partition $\bm{A}(\bm{Q})$ is the one that allocates each cell $(r,s)\in \bm{\Omega}(\bm{Q})$ to the dimensional facility in $\{P_1,...,P_{\rho}\}$ that provides the  minimum assignment cost, i.e., if $(r,s)\in \bm{A}_i(\bm{Q})$ for some $i\in\{1,...,\rho\}$ in the optimal partition, then
$$ a_i w^D_{rs} + w^D_{rs}  u(q_{(r,s)},P_i^{q_{(k_i,l_i)}})\leq  a_j w^D_{rs} + w^D_{rs}  u(q_{(r,s)},P_j^{q_{(k_j,l_j)}}),$$
for all  $j\in\{1,...,\rho\}$. Note that there may exist cells $(r,s)\in \bm{\Omega}(\bm{Q})$ for which
$$ a_i w^D_{rs} + w^D_{rs}  u(q_{(r,s)},P_i^{q_{(k_i,l_i)}}) =  a_j w^D_{rs} + w^D_{rs}  u(q_{(r,s)},P_j^{q_{(k_j,l_j)}})$$
for some $i,j\in \{1,...,\rho\}$ with $i\not=j$, such that they have a non $D$-negligible demand density $w^D_{rs}$. Therefore, in the discrete scheme, we can not define the best reply function $\widehat{\bm{A}}: \bm{Q}\in \bm{\Gamma}\to  \widehat{\bm{A}}(\bm{Q})\in \bm{\mathcal{A}}_{\rho}(\bm{Q})$ as we have done  in the non-discretized problem, since it could be not injective.

Reasoning in the same way as above, problem DBL can be expressed as:
$$ \begin{array}{ll} \displaystyle  \min  & \bm{\mathcal{F}}(\bm{Q}) \\
s.t. &   \widehat{\bm{A}}(\bm{Q})\in \arg\displaystyle  \min_{\bm{A}(\bm{Q})\in \bm{\mathcal{A}}_{\rho}(\bm{Q})} \sum_{i=1}^{\rho}   \sum_{(r,s)\in \bm{A}_i(\bm{Q})} [a_i w^D_{rs} + w^D_{rs}  u(q_{(r,s)},P_i^{q_{(k_i,l_i)}}) ],   \\
     &   \bm{Q} \in \bm{\Gamma},  \end{array}  \eqno{\text{DBL}} $$
being
$$\bm{\mathcal{F}}(\bm{Q}) := \displaystyle   \sum_{i=1}^{\rho} \left[I_i^{\text{PL}}\left(\sum_{(r,s)\in\bm{P}_i^{(k_i,l_i)}} w_{rs}^B\right) + C_i^{\text{PL}}\left(\sum_{(r,s)\in \widehat{\bm{A}}_i(\bm{Q})} w^D_{rs}\right)\right] + L^{\text{PL}}\left(\sum_{i=1}^{\rho} \sum_{(r,s)\in\bm{P}_i^{(k_i,l_i)}} w_{rs}^D\right),$$
where we are using the notation
$w^B_{rs} = \int_{(r,s)} B(q) dq$
for any $(r,s)\in \bm{\Omega}$. Problem DBL is again a bilevel problem since to evaluate a suitable solution $\bm{Q}\in \bm{\Gamma}$ in the objective function $\bm{\mathcal{F}}$ one needs to solve before problem DLL$(\bm{Q})$.  Note that $Q=(q_{(k_1,l_1)},...,q_{(k_{\rho},l_{\rho})})\in \Gamma$ for all $\bm{Q}=((k_1,l_1),...,(k_{\rho},l_{\rho}))\in \bm{\Gamma}$, i.e., every suitable solution of problem DBL codifies a suitable solution of problem BL.

It is easy to prove that problem DBL is NP-hard with a reduction from the $\rho$-median problem, where $\rho$ is the number of facilities to be located in our problem.


In the following, we suppose that there exists a suitable solution $\mathring{Q}=(\mathring{q}_1,...,\mathring{q}_{\rho})\in \Gamma$ for problem BL such that $P_i^{\mathring{q}_i}\cap \partial \Omega = \emptyset$ for all $i\in \{1,...,\rho\}$ and $P_i^{\mathring{q}_i}\cap P_j^{\mathring{q}_j} = \emptyset$ for all $i,j\in \{1,...,\rho\}$ with $i\not=j$. This ensures the existence of a grid $G$, fine enough, for which problem DBL has at least one suitable solution: take a grid $G$ in which the point $\mathring{q}_i$ is the center of one of the cells in $\bm{G}$, say $(\mathring{k}_i,\mathring{l}_i)$, for each $i\in\{1,...,\rho\}$, and fine enough to guarantee $\bm{P}_i^{(\mathring{k}_i,\mathring{l}_i)}\cap \bm{P}_j^{(\mathring{k}_j,\mathring{l}_j)} = \emptyset$ for all $i,j\in \{1,...,\rho\}$ with $i\not=j$; then, $\mathring{\bm{Q}}=((\mathring{k}_1,\mathring{l}_1),...,(\mathring{k}_{\rho},\mathring{l}_{\rho}))\in \bm{\Gamma}$.

Next, we show our convergence results.

Let us consider a sequence of successively  refined grids $\{ G(n) \}_{n\in \mathbb{N}}$ satisfying that $G(1)$ is a grid for which problem DBL has at least one suitable solution. The sequence of grids $\{ G(n) \}_{n\in \mathbb{N}}$ is a sequence of successively  refined grids if given a grid $G(\tilde{n})$ and any of its cell $(\tilde{k},\tilde{l})$, there exists $\breve{n}\in \mathbb{N}$ with $\breve{n} > \tilde{n}$ such that $(\tilde{k},\tilde{l})$ is the union of a set of cells of the grid $G(\breve{n})$ with strictly less width and height than $(\tilde{k},\tilde{l})$. We add an additional index $n$ to the notation introduced in the section to indicate the grid of the sequence which is being considered in each case. For example, DLL$(\bm{Q},n)$ is the discretized lower level problem for a suitable solution $\bm{Q}\in \bm{\Gamma}(n)$ when we consider the grid $G(n)$, $n\in \mathbb{N}$. Finally, we denote by $\kappa (n)$ the maximum edge length of a cell in $\bm{G}(n)$, $n\in \mathbb{N}$.

In the following results, we assume that the functions $I_1^{\text{PL}}(\cdot , n),...,I_{\rho}^{\text{PL}}(\cdot, n),C_1^{\text{PL}}(\cdot, n),...,C_{\rho}^{\text{PL}}(\cdot, n),L^{\text{PL}}(\cdot, n)$ of  problem DBL$(n)$ are obtained from the functions $I_1(\cdot),...,I_{\rho}(\cdot),C_1(\cdot),...,C_{\rho}(\cdot),L(\cdot)$ of problem BL by linear interpolation over a partition of the corresponding domains, in such a way that, the larger the $n$, the finer the partition. Moreover, we suppose that the partition is such that, for any $\varepsilon > 0$, there exists $\tilde{n}\in\mathbb{N}$ such that  $\left|I_1(\omega_1^B)-I_1^{\text{PL}}(\omega_1^B,n)\right| < \varepsilon$, for all $\omega_1^B \in \mathbb{R}$ and all $n\in \mathbb{N}$ with $n\geq \tilde{n}$. The same assumption is done  for the remaining mentioned functions. Note that this assumptions can be done due to the properties assumed for the functions  $I_1,...,I_{\rho},C_1,...,C_{\rho},L$.


\begin{Lem}\label{convergenceI}
Let $i\in\{1,...,\rho\}$. For any $\varepsilon > 0$, there exists $n(\varepsilon)\in \mathbb{N}$ such that
$$\left| I_i\left( \int_{P_i^{q_{(k_i,l_i)}}} B(q) dq \right) - I_i^{\text{PL}}\left( \int_{\mathbb{P}_i^{(k_i,l_i)}(n)} B(q) dq,n \right) \right|< \varepsilon,$$
for all $\bm{Q}=((k_1,l_1),...,(k_{\rho},l_{\rho}))\in \bm{\Gamma}(n)$ and all $n\in \mathbb{N}$ with $n\geq n(\varepsilon)$.
\end{Lem}
\begin{proof}
First, take $\tilde{n}\in\mathbb{N}$ such that $\left|I_i(\omega_i^B)-I_i^{\text{PL}}(\omega_i^B,n)\right| < \varepsilon / 2$, for all $\omega_i^B \in \left[0,\int_{\Omega} B(q) dq \right]$ and all $n\in \mathbb{N}$ with $n\geq \tilde{n}$. Since $I_i$ is continuous, it is uniformly continuous on $\left[0,\int_{\Omega} B(q) dq \right]$, therefore, for $\varepsilon / 2$ there exists $\xi > 0$ such that, when $|\omega_i^B - \tilde{\omega}_i^B| < \xi$, then $|I_i(\omega_i^B) - I_i(\tilde{\omega}_i^B)| < \varepsilon / 2$, for all $\omega_i^B, \tilde{\omega}_i^B \in \left[0,\int_{\Omega} B(q) dq \right]$.

Let $\mathbb{P}^-_i(n)$ be the dimensional facility such that, when it is located at the point $q_i\in\mathbb{R}^2$, it is given by $\mathbb{P}^-_i(q_i,n)=\{q\in \mathbb{R}^2 : q\in P_i^{\tilde{q}_i} \text{ for all } \tilde{q}_i\in \mathcal{B}_{\infty}(q_i,\kappa(n))\}$, for each $n\in\mathbb{N}$ with $n\geq \breve{n}$, being $\breve{n}\in\mathbb{N}$ large enough. In addition, let $\mathbb{P}^+_i(n)$ be the dimensional facility such that, when it is located at the point $q_i\in\mathbb{R}^2$, it is given by $\mathbb{P}^+_i(q_i,n)=\{q\in \mathbb{R}^2 : q\in P_i^{\tilde{q}_i} \text{ for some } \tilde{q}_i\in \mathcal{B}_{\infty}(q_i,\kappa(n))\}$, for each $n\in\mathbb{N}$.
Whereas $\mathbb{P}^+_i(n)$ is a dimensional facility for all $n\in \mathbb{N}$, $\mathbb{P}^-_i(n)$ can not be a dimensional facility for all $n\in \mathbb{N}$. However, it is not difficult to see that $\mathbb{P}^-_i(n)$ is a dimensional facility for all $n$ large enough. This is the reason why we define $\mathbb{P}^-_i(n)$ only  for each $n\in\mathbb{N}$ with $n\geq \breve{n}$, being $\breve{n}\in\mathbb{N}$ large enough.

Note that
$ \mathbb{P}_i^-(q_i,n)\subseteq \mathbb{P}_i^-(q_i,n+1) \subseteq   P_i^{q_i}$, and that $\bigcup_{n\geq \breve{n}}^{\infty} \mathbb{P}_i^-(q_i,n) = P_i^{q_i}$ up to $m$-negligible sets, for any $q_i\in \mathbb{R}^2$. In the same way, $P_i^{q_i} \subseteq \mathbb{P}_i^+(q_i,n+1) \subseteq \mathbb{P}_i^+(q_i,n)$ and $\bigcap_{n\geq \breve{n}}^{\infty} \mathbb{P}_i^+(q_i,n) = P_i^{q_i}$ up to $m$-negligible sets, for any $q_i\in \mathbb{R}^2$. Then, reasoning in the same way that in the proof of Theorem \ref{existenceThm}, it can be shown that  there exists $\bar{n}\in \mathbb{N}$ with $\bar{n}\geq \breve{n}$ such that $\int_{P_i^{q_i} \setminus \mathbb{P}_i^- (q_i,n) } B(q)dq < \xi $ and $\int_{\mathbb{P}_i^+ (q_i,n) \setminus  P_i^{q_i}  } B(q)dq < \xi  $ for all $q_i\in \Omega_i$ and all $n\in\mathbb{N}$ with $n\geq \bar{n}$.

Moreover, it is not difficult to see that $\mathbb{P}_i^-(q_{(k_i,l_i)},n)\subseteq \mathbb{P}_i^{(k_i,l_i)}(n)\subseteq \mathbb{P}_i^+(q_{(k_i,l_i)},n)$, for any $n\in\mathbb{N}$ with $n\geq \breve{n}$ and any suitable solution $\bm{Q}=((k_1,l_1),...,(k_{\rho},l_{\rho}))\in \bm{\Gamma}(n)$ of the problem DBL$(n)$. Therefore, $\left| \int_{P_i^{q_{(k_i,l_i)}}} B(q)dq - \int_{\mathbb{P}_i^{(k_i,l_i)}(n)} B(q)dq\right| < \xi$ for all $\bm{Q}=((k_1,l_1),...,(k_{\rho},l_{\rho}))\in \bm{\Gamma}(n)$ and all $n\in\mathbb{N}$ with $n\geq \bar{n}$.

The proof is completed taking $n(\varepsilon) = \max \{\tilde{n}, \bar{n}\}$. \qed
\end{proof}

\begin{Lem}\label{convergenceC}
Let $i\in\{1,...,\rho\}$. For any $\varepsilon > 0$, there exists $n(\varepsilon)\in \mathbb{N}$ such that
$$\left| C_i\left( \int_{\widehat{A}_i(Q)} D(q) dq \right) - C_i^{\text{PL}}\left( \int_{\widehat{A}_i(\bm{Q},n)} D(q) dq,n \right) \right|< \varepsilon,$$
for all $\bm{Q}=((k_1,l_1),...,(k_{\rho},l_{\rho}))\in \bm{\Gamma}(n)$, being $Q=(q_{(k_1,l_1)},...,q_{(k_{\rho},l_{\rho})})$, and all $n\in \mathbb{N}$ with $n\geq n(\varepsilon)$.
\end{Lem}
\begin{proof}
For all $Q=(q_1,...,q_{\rho})\in \Gamma$ and all $n\in \mathbb{N}$, we define the sets
$\widehat{\mathbb{A}}_i^-(Q,n) = \{q\in \Omega: q\in (\widehat{A}_i(Q))^{\tilde{q}_i}  \text{ for all } \tilde{q}_i\in\mathcal{B}_{\infty}(\widehat{q}_i(Q,n),\kappa(n))\}\setminus \mathbb{P}_i^+(q_i,n)$
and
$\widehat{\mathbb{A}}^+_i(Q,n) = \{q\in \Omega : q\in (\widehat{A}_i(Q))^{\tilde{q}_i}  \text{ for some } \tilde{q}_i\in\mathcal{B}_{\infty}(\widehat{q}_i(Q,n),\kappa(n))\}$, being: $\widehat{q}_i(Q,n)$  any point in $\widehat{A}_i(Q,n)$; $(\widehat{A}_i(Q,n))^{\tilde{q}_i}$  the set $\widehat{A}_i(Q,n)$  when we apply to it the translation induced by the vector $\overrightarrow{\widehat{q}_i(Q,n) \tilde{q}_i}\in \mathbb{R}^2$, for any $\tilde{q}_i\in\mathcal{B}_{\infty}(\widehat{q}_i(Q,n),\kappa(n))$. Note that the definition of the sets above: does not depend on the point $\widehat{q}_i(Q,n)$ chosen; induces two applications  $\widehat{\mathbb{A}}_i^-$ and $\widehat{\mathbb{A}}_i^+$ with domain on $\Gamma \times \mathbb{N}$.

Let $Q\in \Gamma$. Note that $ \widehat{\mathbb{A}}_i^-(Q,n)\subseteq \widehat{\mathbb{A}}_i^-(Q,n+1) \subseteq   \widehat{A}_i(Q)$, and that $\bigcup_{n=1}^{\infty} \widehat{\mathbb{A}}_i^-(Q,n) = \widehat{A}_i(Q)$ up to $m$-negligible sets. In the same way, $ \widehat{A}_i(Q) \subseteq  \widehat{\mathbb{A}}_i^+(Q,n+1)\subseteq \widehat{\mathbb{A}}_i^+(Q,n) \subseteq  $ and $\bigcap_{n=1}^{\infty} \widehat{\mathbb{A}}_i^+(Q,n) = \widehat{A}_i(Q)$ up to $m$-negligible sets. Then, reasoning in the same way that in the proof of Theorem \ref{existenceThm}, it can be shown that, given $\xi > 0$, there exists $\tilde{n}\in \mathbb{N}$ such that $\int_{\widehat{A}_i(Q) \setminus \widehat{\mathbb{A}}^-_i(Q,n)} D(q)dq < \xi $ and $\int_{\widehat{\mathbb{A}}^+_i(Q,n) \setminus \widehat{A}_i(Q)} D(q)dq < \xi  $ for  all $n\in\mathbb{N}$ with $n\geq \tilde{n}$. Moreover, as the statement above is true for all $Q\in \Gamma$, there exists $\breve{n}\in \mathbb{N}$ such that $\int_{\widehat{A}_i(Q) \setminus \widehat{\mathbb{A}}^-_i(Q,n)} D(q)dq < \xi $ and $\int_{\widehat{\mathbb{A}}^+_i(Q,n) \setminus \widehat{A}_i(Q)} D(q)dq < \xi  $ for  all $Q\in \Gamma$ and all  $n\in\mathbb{N}$ with $n\geq \breve{n}$.

It is not difficult to see that $\widehat{\mathbb{A}}^-_i(Q,n)\subseteq \widehat{A}_i(\bm{Q},n)\subseteq \widehat{\mathbb{A}}^+_i(Q,n)$, for all $\bm{Q}=((k_1,l_1),...,(k_{\rho},l_{\rho}))\in \bm{\Gamma}(n)$, being $Q=(q_{(k_1,l_1)},...,q_{(k_{\rho},l_{\rho})})$, and all $n\in \mathbb{N}$. So, at this point, the proof can be completed adapting the one of Lemma \ref{convergenceI}. \qed
\end{proof}

\begin{Lem}\label{convergenceL}
For any $\varepsilon > 0$, there exists $n(\varepsilon)\in \mathbb{N}$ such that
$$\left| L\left( \int_{\bigcup_{i=1}^{\rho}P_i^{q_{(k_i,l_i)}}} D(q) dq \right) - L^{\text{PL}}\left( \int_{\bigcup_{i=1}^{\rho} \mathbb{P}_i^{(k_i,l_i)}(n)} D(q) dq,n \right) \right|< \varepsilon,$$
for all $\bm{Q}=((k_1,l_1),...,(k_{\rho},l_{\rho}))\in \bm{\Gamma}(n)$ and all $n\in \mathbb{N}$ with $n\geq n(\varepsilon)$.
\end{Lem}
\begin{proof}
The proof is similar to the one of Lemma \ref{convergenceI}.\qed
\end{proof}

From these lemmas one can obtain the final convergence result.

\begin{Thm}\label{convergenceDBL}
Suppose that, for any suitable solution $Q\in \Gamma$ of problem BL and for any $\epsilon >0$, there exists $\tilde{Q}=(\tilde{q}_1,...,\tilde{q}_{\rho})\in \mathcal{B}_{\infty}(Q,\epsilon)\cap \Gamma$ such that $P_i^{\tilde{q}_i}\cap \partial \Omega = \emptyset$ for all $i\in \{1,...,\rho\}$ and $P_i^{\tilde{q}_i}\cap P_j^{\tilde{q}_j} = \emptyset$ for all $i,j\in \{1,...,\rho\}$ with $i\not=j$. Then, for any $\varepsilon > 0$, there exists $n(\varepsilon)\in \mathbb{N}$ such that:
\begin{enumerate}
\item $\left| \mathcal{F}(Q^*) - \bm{\mathcal{F}}(\bm{Q}^*,n)\right| < \varepsilon$,

\item $\left| \mathcal{F}(Q^*) - \mathcal{F}(\bar{Q})\right| < \varepsilon,$
\end{enumerate}
for all $n\in \mathbb{N}$ with $n\geq n(\varepsilon)$, being $Q^*$ an optimal suitable solutions of problem BL, $\bm{Q}^*=((k_1^*,l_1^*),...,(k_{\rho}^*,l_{\rho}^*))$ an optimal suitable solutions of problem DBL$(n)$, and $\bar{Q}=(q_{(k_1^*,l_1^*)},...,q_{(k_{\rho}^*,l_{\rho}^*)})$ the suitable solution of problem BL codified by $\bm{Q}^*$.
\end{Thm}
\begin{proof}
From Lemma \ref{convergenceI}, Lemma \ref{convergenceC} and Lemma \ref{convergenceL} is derived that there exists $\tilde{n}\in\mathbb{N}$ such that $\left| \mathcal{F}(Q) - \bm{\mathcal{F}}(\bm{Q},n)\right| < \varepsilon / 4$,
for all $\bm{Q}=((k_1,l_1),...,(k_{\rho},l_{\rho}))\in \bm{\Gamma}(n)$, being $Q=(q_{(k_1,l_1)},...,q_{(k_{\rho},l_{\rho})})$, and all $n\in \mathbb{N}$ with $n\geq \tilde{n}$.

Due to $\mathcal{F}$ is continuous on $\Gamma$ as it was been shown in the proof of Theorem \ref{existenceThm},  there exists $\epsilon > 0$ such that, if $Q\in \mathcal{B}_{\infty}(Q^*,\epsilon) \cap \Gamma$, then $\left| \mathcal{F}(Q^*) - \mathcal{F}(Q)\right| < \varepsilon/ 4$. Moreover, by hypothesis, there exists $\tilde{Q}=(\tilde{q}_1,...,\tilde{q}_{\rho})\in \mathcal{B}_{\infty}(Q^*,\epsilon)\cap \Gamma$ such that $P_i^{\tilde{q}_i}\cap \partial \Omega = \emptyset$ for all $i\in \{1,...,\rho\}$ and $P_i^{\tilde{q}_i}\cap P_j^{\tilde{q}_j} = \emptyset$ for all $i,j\in \{1,...,\rho\}$ with $i\not=j$. It is not difficult to see that then there exists $\breve{n}\in \mathbb{N}$ for which $\tilde{q}_i\in (k_i,l_i)$, for each $i\in\{1,...,\rho\}$, for some $\bm{Q}=((k_1,l_1),...,(k_{\rho},l_{\rho}))\in \bm{\Gamma}(\breve{n})$. Moreover, note that, for all $n\in \mathbb{N}$ with $n\geq \breve{n}$,  there always exists $\bm{Q}=((k_1,l_1),...,(k_{\rho},l_{\rho}))\in \bm{\Gamma}(n)$ such that $\tilde{q}_i\in (k_i,l_i)$ for each $i\in\{1,...,\rho\}$. Using the continuity of  $\mathcal{F}$ on $\Gamma$ and taking into account that  $\{ G(n) \}_{n\in \mathbb{N}}$ is a sequence of successively  refined grids, it can be proven that there exists $\bar{n}\in\mathbb{N}$ with $\bar{n}\geq \breve{n}$ such that $\left| \mathcal{F}(\tilde{Q}) - \mathcal{F}(\breve{Q})\right| < \varepsilon/ 4$, being $\breve{Q}=(q_{(\tilde{k}_1,\tilde{l}_1)},...,q_{(\tilde{k}_{\rho},\tilde{l}_{\rho})})\in\Gamma$ the suitable solution of problem BL codified  by the suitable solution $\tilde{\bm{Q}}=((\tilde{k}_1,\tilde{l}_1),...,(\tilde{k}_{\rho},\tilde{l}_{\rho}))\in \bm{\Gamma}(n)$ of problem DBL$(n)$ verifying $\tilde{q}_i\in (\tilde{k}_i,\tilde{l}_i)$ for each $i\in\{1,...,\rho\}$, for all $n\geq \bar{n}$.

Let $n(\varepsilon) = \max\{ \tilde{n}, \bar{n}\}$. Take $n\in \mathbb{N}$ with $n\geq n(\varepsilon)$ and let $\breve{Q}=(q_{(\tilde{k}_1,\tilde{l}_1)},...,q_{(\tilde{k}_{\rho},\tilde{l}_{\rho})})\in\Gamma$ the suitable solution of problem BL codified  by the suitable solution $\tilde{\bm{Q}}=((\tilde{k}_1,\tilde{l}_1),...,(\tilde{k}_{\rho},\tilde{l}_{\rho}))\in \bm{\Gamma}(n)$ of problem DBL$(n)$ verifying $\tilde{q}_i\in (\tilde{k}_i,\tilde{l}_i)$ for each $i\in\{1,...,\rho\}$. From the reasoning above, $\left| \mathcal{F}(Q^*) - \bm{\mathcal{F}}(\tilde{\bm{Q}},n)\right| < \dfrac{3 \varepsilon}{4}$. If $\bm{Q}^*=((k_1^*,l_1^*),...,(k_{\rho}^*,l_{\rho}^*))\in \bm{\Gamma}(n)$ is the optimal suitable solution solution of problem DBL(n), then $\left| \bm{\mathcal{F}}(\bm{Q}^*,n) - \mathcal{F}(\bar{Q}) \right| < \varepsilon / 4$, being $\bar{Q}=(q_{(k_1^*,l_1^*)},...,q_{(k_{\rho}^*,l_{\rho}^*)})$. Now, observe that, if $\bm{\mathcal{F}}(\bm{Q}^*,n) \leq \mathcal{F}(Q^*)$, then  $\bm{\mathcal{F}}(\bm{Q}^*,n) \leq \mathcal{F}(Q^*) \leq \mathcal{F}(\bar{Q})$, which implies $\left| \mathcal{F}(Q^*) - \bm{\mathcal{F}}(\bm{Q}^*,n)\right| < \varepsilon / 4 < \varepsilon$. On the other hand, if $ \mathcal{F}(Q^*) \leq \bm{\mathcal{F}}(\bm{Q}^*,n)$, then  $ \mathcal{F}(Q^*) < \bm{\mathcal{F}}(\bm{Q}^*,n) < \bm{\mathcal{F}}(\tilde{\bm{Q}},n)$, which implies $\left| \mathcal{F}(Q^*) - \bm{\mathcal{F}}(\bm{Q}^*,n)\right| < \dfrac{3 \varepsilon}{4}< \varepsilon$. Finally, taking into account the above, it is not difficult to see that $\left| \mathcal{F}(Q^*) - \mathcal{F}(\bar{Q})\right| < \varepsilon$.\qed
\end{proof}

The theorem above proves the convergence of the sequence of solutions for the discrete approximation to the optimal objective value of problem BL.

\section{Solution approaches}

Section 3 provides a methodology to solve problem BL by sequences of discrete problems DBL that converge to the optimal objective value. However, solving each one of those discrete approximations is an issue by itself, but, as we will see in the following, we propose two methods to solve the problem DBL: one of them is exact and it consists of a mixed-integer linear programming (MILP) model and the other one is a GRASP heuristic (see \cite{graspRF}).

\subsection{A mathematical programming formulation}

This section provides a valid MILP formulation for problem DBL for a fixed grid $G$.

In order to give a valid formulation for problem DBL we need to determine the sets and parameters that charge the model with the necessary information of the problem. At this point we remark that the overall global computation time to get an optimal solution of problem DBL is the computing time to obtain the input sets and parameters of the model plus the computing time required to reach the optimal solution. Our goal is to get a solution time as small as possible, so that we have to properly balance both times. On the one hand, if we do not preprocess adequately the information from the elements of the problem, then the model will have to work too much to obtain that information and, as it is known, this is not desirable since MILP models can be really hard to solve. On the other hand, if we want to fully preprocess the elements of the problem to do the model work less, we will have to do different operations over the set of cells of $\bm{\Omega}$. Since we are interesting in $|\bm{\Omega}|$ (number of cells of $\bm{\Omega}$) to be large (to better approximate problem BL by problem DBL), the time to obtain the initial information sets and CPU memory consumption can increase dramatically.

We use the following sets and parameters to build our MILP model:
\begin{enumerate}[-]
\item $\bm{\Omega}_i$: set of candidates for feasible location of dimensional facility $P_i$ in problem DBL, i.e., the set of cells $(k,l)\in\bm{\Omega}$ such that $P_i^{q_{(k,l)}}\subseteq \Omega$. This set is defined for each $i\in \{1,...,\rho\}$.

\item $\bm{E}_{rs}^i$: set of cells $(k,l)$ in $\bm{\Omega}_i$ verifying $(r,s)\in \bm{P}_i^{(k,l)}$. We define this set for each $(r,s)\in \bm{\Omega}$ and $i\in \{1,...,\rho\}$.

\item $w^D_{rs}$: the demand density in the cell $(r,s)$. This parameter is defined for each $(r,s)\in \bm{\Omega}$.

\item $w^B_{rs}$: the base installation cost density in the cell $(r,s)$. This parameter is defined for each $(r,s)\in \bm{\Omega}$.

\item $u_{rs,kl}^{i}$: $u(q_{(r,s)},P_i^{q_{(k,l)}})$, i.e., the utility in problem DBL obtained from any point in $(r,s)$ with respect to the dimensional facility $P_i$ when its root point is located at the center of the cell $(k,l)$. If $(r,s)\in \bm{P}^{(k,l)}_i$ we take $u_{rs,kl}^{i}=-a_i$ (the reason of this choice will be easily understood when the model is presented). We define this parameter for each $(r,s)\in \bm{\Omega}$, $i\in\{1,...,\rho\}$ and $(k,l)\in \bm{\Omega}_i$.
\end{enumerate}

We now analyze the asymptotic computational complexity for obtaining these sets and parameters assuming $\bm{\Omega}$ has already been determined. For each $i\in \{1,...,\rho\}$, suppose that $\mathcal{O}(f_1(P_i,\Omega))$ is the asymptotic computational complexity bound for testing if the dimensional facility $P_i$ with its root point fixed at a point $q_i\in \mathbb{R}^2$ satisfies $P_i^{q_i}\subseteq \Omega$. Then, obtaining $\bm{\Omega}_i$ can be done in $\mathcal{O}(|\bm{\Omega}|f_1(P_i,\Omega))$ (one check for each point $q_{(r,s)}$ with $(r,s)\in \bm{\Omega}$). Thus,  the complexity to get all the sets $\{\bm{\Omega}_1,...,\bm{\Omega}_{\rho}\}$ is bounded by $\mathcal{O}(|\bm{\Omega}|\sum_{i=1}^{\rho}f_1(P_i,\Omega))$.

For each $i\in \{1,...,\rho\}$, once $\bm{\Omega}_i$ is computed, take $(k,l)\in \bm{\Omega}_i$. For each $(r,s)\in \bm{\Omega}$ check if $\text{int}((r,s))\cap \text{int}(P_i^{q_{(k,l)}})\not= \emptyset$ and let $\mathcal{O}(f_2(P_i))$ be the time required to do that test for the cell. If $\text{int}((r,s))\cap \text{int}(P_i^{q_{(k,l)}})\not= \emptyset$, add $(k,l)$ to $\bm{E}^i_{rs}$ and take $u^i_{rs,kl}=-a_i$. Otherwise, compute $u(q_{(r,s)},P_i^{q_{(k,l)}})$ and take $u^i_{rs,kl}=u(q_{(r,s)},P_i^{q_{(k,l)}})$. Let $\mathcal{O}(f_3(P_i))$ be the complexity for computing the utility $u(q,P_i^{q_i})$ for any $q,q_i\in \mathbb{R}^2$. Hence, the asymptotic computational complexity of obtaining all the sets $\bm{E}^i_{rs}$ and all the paremeters $u^i_{rs,kl}$ can be bounded by   $\mathcal{O}(|\bm{\Omega}|^2\sum_{i=1}^{\rho} [f_2(P_i) + f_3(P_i)])$ ($|\bm{\Omega}_i|$ is at most $|\bm{\Omega}|$).

As for the parameters $w^D_{rs}$, to obtain all of them it is necessary to compute  $|\bm{\Omega}|$ integrals. The same can be said for parameters $w^B_{rs}$.

The above analysis shows that all the sets and parameters which we use to define the MILP model can be obtained in a ``reasonable'' computation time. The space requirements are also efficient and can be bounded above by: $\sum_{i=1}^{\rho} |\bm{\Omega}_i|\leq \rho|\bm{\Omega}|$, $ \sum_{i=1}^{\rho}  \sum_{(r,s)\in \bm{\Omega}}|\bm{E}_{rs}^i|\leq  \rho|\bm{\Omega}|^2$, there are $|\bm{\Omega}|$ constants $w^D_{rs}$, the same number of parameters  $w^B_{rs}$, and the cardinality of $u^i_{rs,kl}$ is at most  $\rho|\bm{\Omega}|^2$.


Next, we describe the MILP model. Recall that any non-decreasing, bounded, continuous, piecewise linear function can be modeled with a MILP formulation, see, for example, \cite{PLFourer}. Below we represent by $\overline{I_1^{\text{PL}}}(\omega_1^I),...,\overline{I_{\rho}^{\text{PL}}}(\omega_{\rho}^I),\overline{C_1^{\text{PL}}}(\omega_1^C),...,\overline{C_{\rho}^{\text{PL}}}(\omega_{\rho}^C),\overline{L^{\text{PL}}}(\omega^L)$ the linearization of the functions $I_1^{\text{PL}}(\omega_1^I),...,I_{\rho}^{\text{PL}}(\omega_{\rho}^I),C_1^{\text{PL}}(\omega_{1}^C),...,C_{\rho}^{\text{PL}}(\omega_{\rho}^C),L^{\text{PL}}(\omega^L)$
in the objective function of a suitable MILP formulation, and by $SCDVPL$ the Set of Constraints and the Domain declaration of the decision Variables involved in the model that together makes the representation of the Piecewise Linear functions to be correct.

In order to understand the model, we define the following families of decision variables. Binary variable $\theta^i_{kl}$ is a location variable: it takes the value $1$ if the root point $p_i$ of the dimensional facility $P_i$ is located at the center of the cell $(k,l)\in \bm{\Omega}_i$, and $0$ otherwise, for each $i\in \{1,...,\rho\}$. Binary variable $\tau^i_{rs}$ is an allocation variable and it takes the value $1$ if customers in the cell $(r,s)\in \bm{\Omega}$ are served by the dimensional facility $P_i$, and $0$ otherwise, for each $i\in \{1,...,\rho\}$. Variable $\varphi_{rs}$ will assume the value of the utility $u_G$ obtained from the cell $(r,s) \in \bm{\Omega}$ when it is assigned to its dimensional facility in a solution of problem DBL. We point out that the facility assigned to a cell must be the one given by a solution of the corresponding discretized lower level problem. Variable $\varphi_{rs}$ will be $0$ if $(r,s)$ is contained in a cell facility, for each $(r,s)\in \bm{\Omega}$.

\begin{Thm} Problem DBL is equivalent to the following MILP problem:
\begin{small}
\begin{align}
\min &  \,\,\,\,\, \displaystyle \sum_{i=1}^{\rho} \overline{I_i^{\text{PL}}}\left( \sum_{(r,s)\in \bm{\Omega}}  \sum_{(k,l)\in \bm{E}^i_{rs}} w^B_{rs}  \theta^i_{kl} \right) + \sum_{i=1}^{\rho} \overline{C_i^{\text{PL}}}\left( \sum_{(r,s)\in \bm{\Omega}} w^D_{rs} \tau^i_{rs} \right) \label{A1}\\
&  \,\,\,\,\, \displaystyle +  \overline{L^{\text{PL}}}\left(\sum_{(r,s)\in \bm{\Omega}} w^D_{rs} \left[1- \sum_{i=1}^{\rho}\tau^i_{rs}\right]\right) \nonumber \\
\text{s.t.}   &  \,\,\,\,\, SCDVPL, \label{A2}\\
&  \,\,\,\,\, \sum_{(k,l)\in \bm{\Omega}_i} \theta^i_{kl} = 1, &&  \forall i\in\{1,...,\rho\},  \label{A3}\\
            &  \,\,\,\,\, \sum_{i=1}^{\rho} \tau^i_{rs} + \sum_{i=1}^{\rho}\sum_{(k,l)\in \bm{E}^i_{rs}} \theta^i_{kl}=1, && \forall (r,s)\in \bm{\Omega}, \label{A4}\\
            &  \,\,\,\,\, \sum_{j=1}^{\rho} a_j w^D_{rs} \tau_{rs}^j + w^D_{rs} \varphi_{rs} \leq a_i w^D_{rs} + \sum_{(k,l)\in \bm{\Omega}_i} w^D_{rs} u^i_{rs,kl} \theta^i_{kl}, && \forall (r,s)\in \bm{\Omega},i\in\{1,...,\rho\}, \label{A5}\\
            &  \,\,\,\,\, \sum_{(k,l)\in \bm{\Omega}_i} u^i_{rs,kl} \theta^i_{kl}-M(1-\tau^i_{rs}) \leq \varphi_{rs} \leq \sum_{(k,l)\in \bm{\Omega}_i} u^i_{rs,kl} \theta^i_{kl}+M(1-\tau^i_{rs}), && \forall (r,s)\in \bm{\Omega},i\in\{1,...,\rho\}, \label{A6}\\
            &  \,\,\,\,\, \theta^i_{kl}\in \{0,1\}, && \forall i\in\{1,...,\rho\},(k,l)\in \bm{\Omega}_i, \label{A7}\\
            &  \,\,\,\,\, \tau^i_{rs}\in \{0,1\}, && \forall (r,s)\in \bm{\Omega},i\in\{1,...,\rho\}, \label{A8}\\
            &  \,\,\,\,\, \varphi_{rs}\geq 0, && \forall (r,s)\in\bm{\Omega}, \label{A9}
\end{align}
\end{small}
\noindent where $M\gg 0$ is a constant large enough.
\end{Thm}
\begin{proof}
First of all, note that the domain of the decision variables is stated in (\ref{A7})-(\ref{A9}).

Suppose that $\bm{Q}=((k_1,l_1),...,(k_{\rho},l_{\rho}))\in \bm{\Gamma}$ is the suitable solution of problem DBL given by the formulation (\ref{A1})-(\ref{A9}). Then, for each $i\in\{1,...,\rho\}$, $\theta_{k_il_i}^i=1$ and $\theta_{kl}^i=0$ for all $(k,l)\in \bm{\Omega}_i$ other than $(k_i,l_i)$, so
$$\sum_{(r,s)\in \bm{\Omega}}  \sum_{(k,l)\in \bm{E}^i_{rs}} w^B_{rs}  \theta^i_{kl} = \sum_{(r,s)\in\bm{P}_i^{(k_i,l_i)}} w_{rs}^B.$$
Moreover, if partition of $\bm{\Omega}(\bm{Q})$ in problem (\ref{A1})-(\ref{A9}) is done according to $\bm{A}(\bm{Q})\in \bm{\mathcal{A}}_{\rho}(\bm{Q})$, it follows that
$$ \sum_{(r,s)\in \bm{\Omega}} w^D_{rs} \tau^i_{rs} = \sum_{(r,s)\in \bm{A}_i(\bm{Q})} w^D_{rs}, $$
since $\tau^i_{rs}$ will be $1$ iff $(r,s)\in \bm{A}_i(\bm{Q})$ for each $(r,s)\in \bm{\Omega}$. This last condition also implies that, for each $(r,s)\in \bm{\Omega}$, $\tau^i_{rs}=0$ for all $i\in \{1,...,\rho\}$ iff $(r,s)\in \bm{\Omega} \setminus \bm{\Omega}(\bm{Q})=\{\bm{P}_1^{(k_1,l_1)}\cup ... \cup \bm{P}_{\rho}^{(k_{\rho},l_{\rho})}\}$, therefore
$$\displaystyle \sum_{(r,s)\in \bm{\Omega}} w^D_{rs} \left[1- \sum_{i=1}^{\rho}\tau^i_{rs}\right] = \sum_{(r,s)\in \bigcup_{i=1}^{\rho} \bm{P}_i^{(k_i,l_i)} } w^D_{rs} = \sum_{i=1}^{\rho} \sum_{(r,s)\in\bm{P}_i^{(k_i,l_i)}} w_{rs}^D.$$
The objective function (\ref{A1}) of the problem (\ref{A1})-(\ref{A9}) minimizes the same function as in problem DBL, given that $\overline{I_1^{\text{PL}}},...,\overline{I_{\rho}^{\text{PL}}},\overline{C_1^{\text{PL}}},...,\overline{C_{\rho}^{\text{PL}}},\overline{L^{\text{PL}}}$ and $SCDVPL$ in (\ref{A2}) are a correct representation, respectively, of $I_1^{\text{PL}},...,I_{\rho}^{\text{PL}},C_1^{\text{PL}},...,C_{\rho}^{\text{PL}},L^{\text{PL}}$. It remains to see that the solution given by the formulation (\ref{A1})-(\ref{A9}) is a suitable solution $\bm{Q}\in\bm{\Gamma}$ of problem DBL and that it provides an optimal partition $\bm{A}(\bm{Q})$ of $\bm{\Omega}(\bm{Q})$ in the corresponding discretized lower level problem.

\indent Constraints (\ref{A3}) state that the root point $p_i$ of the dimensional facility $P_i$ has to be set in one of the cells of the set $\bm{\Omega}_i$ of candidates for feasible location of the dimensional facility $P_i$ in problem DBL, for each $i\in \{1,...,\rho\}$.

\indent With constraints (\ref{A4}), several conditions are imposed. On the one hand,  (\ref{A4}) implies $\sum_{i=1}^{\rho} \tau^i_{rs} \leq 1$, so demand of the cell $(r,s)\in \bm{\Omega}$ can not be satisfied by more than one dimensional facility. On the other hand, implication $\sum_{i=1}^{\rho}\sum_{(k,l)\in \bm{E}^i_{rs}} \theta^i_{kl}\leq 1$ of (\ref{A4}) avoids intersections amongs the interiors of the cell facilities located according to the variables $\theta^i_{kl}$. Suppose that the root points of the dimensional facilities $P_i$ and $P_j$ have been fixed at the centers of the cells $(k_i,l_i)\in \bm{\Omega}_i$ and $(k_j,l_j)\in \bm{\Omega}_j$ respectively, so $\theta^i_{k_il_i}=1$ and $\theta^j_{k_jl_j}=1$, $i,j\in \{1,...,\rho\}$, $i\not=j$. If $\bm{P}^{(k_i,l_i)}_i\cap \bm{P}^{(k_j,l_j)}_j\not= \emptyset$, then there exists $(r,s)\in \bm{\Omega}$ such that $(r,s)\in \bm{P}^{(k_i,l_i)}_i$ and $(r,s)\in \bm{P}^{(k_j,l_j)}_j$, and therefore $(k_i,l_i)\in \bm{E}^i_{rs}$ and $(k_j,l_j)\in \bm{E}^j_{rs}$. This implies $\sum_{i=1}^{\rho}\sum_{(k,l)\in \bm{E}^i_{rs}} \theta^i_{kl}\geq 2$ which contradicts implication $\sum_{i=1}^{\rho}\sum_{(k,l)\in \bm{E}^i_{rs}} \theta^i_{kl}\leq 1$ of (\ref{A4}). Also, constraints (\ref{A4}) force demand of cell $(r,s)\in\bm{\Omega}$ to be satisfied by one dimensional facilty if $(r,s)$ does not belong to any cell facility ($\sum_{i=1}^{\rho} \tau^i_{rs} = 1$ and $\sum_{i=1}^{\rho}\sum_{(k,l)\in \bm{E}^i_{rs}} \theta^i_{kl} = 0$), and to belong to a cell facility if its demand is not satisfied by any dimensional facility ($\sum_{i=1}^{\rho} \tau^i_{rs} = 0$ and $\sum_{i=1}^{\rho}\sum_{(k,l)\in \bm{E}^i_{rs}} \theta^i_{kl} = 1$). So, constraints (\ref{A5})-(\ref{A6}) ensure the feasible location of the dimensional facilities and makes a distinction between demand cells and cells contained in the cell facilities.

\indent The correct allocation of demand cells to dimensional facilities according to the corresponding discretized lower level problem is achieved with constraints (\ref{A5}) and (\ref{A6}). Indeed, suppose that, for $(r,s)\in \bm{\Omega}$, $\tau^{\tilde{i}}_{rs}=1$ for some $\tilde{i}\in\{1,...,\rho\}$. Hence, by constraints (\ref{A4}), $\tau^{j}_{rs}=0$ for all  $j\in \{1,...,\rho\}$ with $j\not= \tilde{i}$, and thus constraints (\ref{A5}) state that $a_{\tilde{i}} w^D_{rs} + w^D_{rs}\varphi_{rs} \leq a_i w^D_{rs} + \sum_{(k,l)\in \bm{\Omega}_i} w^D_{rs} u^i_{rs,kl} \theta^i_{kl}$ for all $i\in \{1,...,\rho\}$. Note that,  as $\sum_{(k,l)\in \bm{\Omega}_i}  \theta^i_{kl} = 1$  by constraints (\ref{A3}), hence $\sum_{(k,l)\in \bm{\Omega}_i}  \theta^i_{kl} =\theta_{k_i,l_i}^i=1$, and we know $(r,s)\notin \bm{P}_i^{(k_i,l_i)}$  due to constraints (\ref{A4}), then $\sum_{(k,l)\in \bm{\Omega}_i} w^D_{rs} u^i_{rs,kl} \theta^i_{kl} = u(q_{(r,s)}, P_i^{q_{(k_i,l_i)}})$, for each $i\in\{1,...,\rho\}$. So, if $\varphi_{rs}$ takes the value $u(q_{(r,s)},P_{\tilde{i}}^{q_{(k_{\tilde{i}},l_{\tilde{i}})}})$, constraints (\ref{A5}) impose that cell $(r,s)$ is assigned to the dimensional facility that provides the smallest cost in the discretized lower level problem. However, by the constraint of type (\ref{A6}) for $(r,s)$ and $i= \tilde{i}$, $\varphi_{rs} = \sum_{(k,l)\in \bm{\Omega}_{\tilde{i}}} w^D_{rs} u^{\tilde{i}}_{rs,kl} \theta^{\tilde{i}}_{kl} = u(q_{(r,s)},P_{\tilde{i}}^{q_{(k_{\tilde{i}},l_{\tilde{i}})}})$ as $M(1-\tau^{\tilde{i}}_{rs}) = 0$. Constraints of type (\ref{A6}) for $(r,s)$ when $i\not= \tilde{i}$ are satisfied trivially as $\sum_{(k,l)\in \bm{\Omega}_i} u^i_{rs,kl} \theta^i_{kl}-M \leq \varphi_{rs}$ and $\varphi_{rs} \leq \sum_{(k,l)\in \bm{\Omega}_i} u^{i}_{rs,kl} \theta^i_{kl}+M$ for $M\gg 0$ large enough.

Now, suppose that for $(r,s)\in \bm{\Omega}$, $\tau^{i}_{rs}=0$ for all $i\in\{1,...,\rho\}$. Then, constraints (\ref{A6}) for $(r,s)$ are satisfied trivially for $M\gg 0$ large enough. As  $\tau^{i}_{rs}=0$ for all $i\in\{1,...,\rho\}$, by constraints (\ref{A4}), we know $\theta_{k_{\tilde{i}}l_{\tilde{i}}}^{\tilde{i}}=1$ for one $\tilde{i}\in \{1,...,\rho\}$ and one $(k_{\tilde{i}},l_{\tilde{i}})\in\bm{\Omega}_{\tilde{i}}$, i.e., $(r,s)\in \bm{P}_{\tilde{i}}^{(k_{\tilde{i}},l_{\tilde{i}})}$. So, constraint of type (\ref{A5}) for $(r,s)$ and $i= \tilde{i}$ imposes $\sum_{j=1}^{\rho} a_j w^D_{rs} \tau_{rs}^j + w^D_{rs}\varphi_{rs} \leq a_{\tilde{i}} w^D_{rs} + \sum_{(k,l)\in \bm{\Omega}_{\tilde{i}}} w^D_{rs} u^{\tilde{i}}_{rs,kl} \theta^{\tilde{i}}_{kl}$. But  $\sum_{j=1}^{\rho} a_j w^D_{rs} \tau_{rs}^j = 0$ and $\sum_{(k,l)\in \bm{\Omega}_{\tilde{i}}} w^D_{rs} u^{\tilde{i}}_{rs,kl} \theta^{\tilde{i}}_{kl} = w^D_{rs} u^{\tilde{i}}_{rs,k_{\tilde{i}}l_{\tilde{i}}} = -w^D_{rs} a_{\tilde{i}}$, according to the definition of parameters $u^i_{rs,kl}$, therefore, $\varphi_{rs}$ has to be $0$. Constraints of type (\ref{A5}) for $(r,s)$ when $i\not= \tilde{i}$ are satisfied trivially as $\sum_{j=1}^{\rho} a_j w^D_{rs} \tau_{rs}^j + w^D_{rs}\varphi_{rs}=0$.

From the above discussion, constraints (\ref{A5})-(\ref{A6}) force the minimum  cost assignment of  cell $(r,s)$ to a dimensional facility in $\{P_1,...,P_{\rho}\}$ for each cell $(r,s)\in \bm{\Omega}$ that is not contained in any cell facility. Hence, the constrains imposed by the discretized lower level problem, in the constrained optimization problem  DBL, are satisfied  for  any feasible solution of the problem (\ref{A1})-(\ref{A9}). So we conclude that problem (\ref{A1})-(\ref{A9}) is equivalent to problem DBL. \qed
\end{proof}

\indent It is not difficult to see that $M=\max \{u^i_{rs,kl} : (r,s)\in \bm{\Omega},i\in\{1,...,\rho\},(k,l)\in \bm{\Omega}_i\}$ is the minimum value of $M$ that makes the above model (\ref{A1})-(\ref{A9}) to be correct.

\subsection{Heuristic method}

As mentioned above, problem DBL is NP-hard, therefore one can not expect to solve large instances with the MILP
formulation (\ref{A1})-(\ref{A9}) which has $2\rho |\bm{\Omega}|$ binary variables defined in (\ref{A7})-(\ref{A8}) plus the number of binary variables in (\ref{A2}) required to modelling the piecewise linear cost functions in (\ref{A1}). This makes the model difficult to solve, especially when the considered number of cells $|\bm{\Omega}|$ is large to better approximate problem BL. For this reason, we introduce an alternative heuristic algorithm to get  ``good/reasonable" feasible solutions of problem DBL for larger size instances.

\indent The algorithm proposed is a GRASP in which we can distinguish three modules. The first module GRASP\_DIMFAC is actually the GRASP, which uses the next two modules to build the final solution. From a location $(q_1,...,q_{\rho})\in \OOmega_1\times ... \times \OOmega_{\rho}$ of the root points of the closed sets $P_1,...,P_{\rho}$ (not necessarily feasible), the second module WAVE\_DIMFAC, which is a continuous wavefront algorithm, generates a random feasible solution $((k_1,l_1),...,(k_{\rho},l_{\rho}))\in\bm{\Gamma}$ of problem DBL. Finally, the third module GREEDY\_DIMFAC is a greedy algorithm that, given a feasible solution $((k_1,l_1),...,(k_{\rho},l_{\rho}))\in\bm{\Gamma}$ of problem DBL, locally searches for another feasible solution improving the objective value of the first one.

\indent In what follows, and for the sake of simplicity, we consider that $\Omega$, the closed sets $P_1,...,P_{\rho}$ and the grid $G$ are fixed. This implies that all the elements that are derived from them are also fixed.

\subsubsection{GRASP algorithm}

Before describing the GRASP, we observe the following. Given a suitable solution $\bm{Q}=((k_1,l_1),...,(k_{\rho},l_{\rho}))\in\bm{\Gamma}$ for problem DBL, computing its objective value can be done easily. This is due to the fact that for each cell $(r,s)\in\bm{\Omega}$, we can know if it is contained in a cell facility, and in which, or if it is a demand cell. If $(r,s)$ is a demand cell we also know to which dimensional facility it is assigned: the one with minimum assignment cost. In other words, $\bm{P}_1^{(k_1,l_1)}, ... ,\bm{P}_{\rho}^{(k_{\rho},l_{\rho})}$ and $\bm{A}(\bm{Q})=(\bm{A}_1(\bm{Q}),...,\bm{A}_\rho(\bm{Q}))$ can be easily obtained processing sequentially all the cells of $\bm{\Omega}$. So, obtained the above sets, we can compute $\bm{\mathcal{F}}(\bm{Q})$.

The above is correct except for the case in which two or more dimensional facilities provide the minimum assigment cost for a cell $(r,s)\in \bm{\Omega}$. In that case, as we are looking for a heuristic solution for problem DBL and we want to do this as fast as possible, we assign the cell $(r,s)$ to any of that dimensional facilities with minimum assignment cost.

A formal pseudocode of our GRASP is given in Algorithm \ref{algGRASP}.

\begin{algorithm}
  \caption{GRASP algorithm for problem DBL} \label{algGRASP}
  \begin{algorithmic}[]
  \State \textbf{PROCEDURE} GRASP\_DIMFAC
  \State \indent \textbf{STEP 1} Create a list $\Psi$ of $\psi\in \mathbb{N}$ suitable solutions for problem DBL as follows. For each $j\in\{1,...,\psi\}$, randomly generate points $q_1\in\OOmega_1,...,q_{\rho}\in\OOmega_{\rho}$ and do $\Psi(j)=\text{GREEDY\_DIMFAC}(\text{WAVE\_DIMFAC}(q_1,...,q_{\rho}))$. Then, order $\Psi$ so that $\bm{\mathcal{F}}(\Psi(j))\leq \bm{\mathcal{F}}(\Psi(j+1))$ for any $j\in\{1,...,\psi-1\}$.
  \State \indent \textbf{STEP 2} Process $\Psi$ visiting its elements from $\Psi(1)$ to $\Psi(\psi)$. For each $\Psi(j)\in \Psi$ do:
  \State \indent \indent $\bullet$ If $\Psi(j)=((k_1,l_1),...,(k_{\rho},l_{\rho}))$, consider the $\rho$-tuple $(q_{(k_1,l_1)},...,q_{(k_{\rho},l_{\rho})})$ and generate a new $\rho$-tuple   $(\tilde{q}_{1},...,\tilde{q}_{\rho})$ ran-
  \State \indent domly permuting exactly $\varpi\in \{2,..., \psi\}$ of its elements.
  \State \indent \indent $\bullet$ Obtain a new suitable solution $\bm{Q}$ doing $\bm{Q}=\text{GREEDY\_DIMFAC}(\text{WAVE\_DIMFAC}(\tilde{q}_{1},...,\tilde{q}_{\rho})$.
  \State \indent \indent $\bullet$ If $\bm{\mathcal{F}}(\bm{Q})<\bm{\mathcal{F}}(\Psi(\psi))$, update $\Psi$ doing $\Psi(\psi)=\bm{Q}$ and reorder the list.
  \State \indent \indent $\bullet$ If this instruction has been visited a maximum number of times, go to RETURN.
  \State \indent \indent $\bullet$  If $j+1=\psi+1$, begin STEP 2 again.
  \State \indent Note: $\text{WAVE\_DIMFAC}(\tilde{q}_1,...,\tilde{q}_{\rho})$ requires $\tilde{q}_i\in\OOmega_i$ for all $i\in\{1,...,\rho\}$ to work. In order to apply the procedure WAVE\_DIMFAC, if $\tilde{q}_{i}\notin\OOmega_i$ for some $i\in\{1,...,\rho\}$, we replace $\tilde{q}_i$ by its $\ell_1$-projection  in $\OOmega_i$.
  \State \indent \textbf{RETURN} $\Psi(1)$ and $\bm{\mathcal{F}}(\Psi(1))$.
  \State \textbf{END PROCEDURE}
  \end{algorithmic}
\end{algorithm}

Our GRASP algorithm for problem DBL takes advantage of the fact we have a tool to generate and evaluate suitable solutions. Initially, in STEP 1, GRASP\_DIMFAC generates a list $\Psi$ of $\psi$ random suitable solutions with procedure WAVE\_DIMFAC and improves them with procedure GREEDY\_DIMFAC. These suitable solutions are ordered in the list $\Psi$ according with their objective values, being the best suitable solution the first in the list.

The randomization part of the GRASP in STEP 2 tries to obtain new suitable solutions from some already available suitable solutions. It performs permutations among the root points of some dimensional facilities ($\varpi$ dimensional facilities, being $\varpi$ a parameter). Given a suitable solution in the list we obtain another one using WAVE\_DIMFAC. This suitable solution may not have the resulting permuted root points  since permuting the positions of the dimensional facilities $P_1,...,P_{\rho}$ in a suitable solution of problem DBL may not provide  another suitable solution, as the interior of the dimensional facilities could intersect or they could not be contained in $\Omega$. Next, we improve that suitable solution with GREEDY\_DIMFAC. If the resulting suitable solution is better than any in the list, we replace the worst suitable solution by the new one, reorder the suitable solutions in the list, and continue the process with the following not yet processed suitable solution in the list. The process is repeated, starting the list by the beginning again if it is necessary, a predefined number of times: termination criterion. The algorithm returns the first element in $\Psi$, i.e., the best suitable solution found for problem DBL, and its objective value.

\subsubsection{Wavefront algorithm}

The main idea of the wavefront algorithm to generate random suitable solutions for problem DBL is the following: since directly locating dimensional facilities $P_1,...,P_{\rho}$ in the demand region in a valid way (i.e., in a way such that its interiors do not intersect) could not be an easy task, we begin by locating in a valid way a shrunken version of them, which is easier, and then we make these shrunken dimensional facilities to grow. The wavefront is shown in Algorithm \ref{algWave}. If $\mathscr{P}_i^{q_i}$ is the homothecy of center $q_i$ and ratio $\lambda\geq 0$ applied to the set $P_i^{q_i}$, in the algorithm, we characterize the location of the set $\mathscr{P}_i^{q_i}$ by the root point $p_i=q_i$, for each $i\in \{1,...,\rho\}$. In addition, we use the following notation in the algorithm: $\ell_1(\mathscr{P}_i^{q_i},\mathscr{P}_j^{q_j})$ is the minimum $\ell_1$-distance  between a point in $\mathscr{P}_i^{q_i}$ and a point in $\mathscr{P}_j^{q_j}$; $\kappa x$ and $\kappa y$ denote the maximum width and  the maximum height of a cell in $\bm{\Omega}$, respectively.

\begin{algorithm}
  \caption{Wavefront algorithm to generate  random suitable solutions for problem DBL} \label{algWave}
  \label{algGreedy}
  \begin{algorithmic}[]
  \State \textbf{PROCEDURE} WAVE\_DIMFAC($q_1,...,q_{\rho}$)
  \State \indent \textbf{STEP 1} For each $i\in\{1,...,\rho\}$, let $\mathscr{P}_i^{q_i}$ be the homothecy of center $q_i$ and ratio $\lambda\in (0,1)$ with $1/\lambda \in\mathbb{N}$ applied to the set $P_i^{q_i}$.
  \State \indent \textbf{STEP 2} If $\{\mathscr{P}_1^{q_1},...,\mathscr{P}_{\rho}^{q_{\rho}}\} = \{P_1^{q_1},...,P_{\rho}^{q_{\rho}}\}$, go to STEP 6. Otherwise:
  \State \indent \indent $\bullet$ Check if $\ell_1(\mathscr{P}_i^{q_i},\mathscr{P}_j^{q_j})\geq 3(\kappa x + \kappa y)$ for all $i,j\in \{1,...,\rho\}$ with $i\not=j$.
  \State \indent \indent \indent $-$ If it is verified, for each $i\in \{1,...,\rho\}$, replace $\mathscr{P}_i^{q_i}$ by the homothecy of center $q_i$ and ratio $\lambda$ applied to it, and begin\State \indent \indent STEP 2 again.
  \State \indent \indent \indent $-$ If it is not verified, go to STEP 3.
  \State \indent \textbf{STEP 3} For each pair $i,j\in \{1,...,\rho\}$ with $i\not=j$, let
  $$\overrightarrow{\upsilon_{ij}} = \begin{cases} \dfrac{\overrightarrow{q_i q_j} }{\Vert \overrightarrow{q_i q_j} \Vert_2} & \hbox{if $\ell_1(\mathscr{P}_i^{q_i},\mathscr{P}_j^{q_j}) < 3(\kappa x + \kappa y)$,} \\ \underline{0} & \hbox{otherwise.} \end{cases}$$
   \State \indent  \textbf{STEP 4}  Do:
   \State \indent \indent $\bullet$ For each $i\in \{1,...,\rho\}$, if $q_i+\vartheta\dfrac{\sum_{j=1}^\rho \overrightarrow{\upsilon_{ji}}}{\Vert \sum_{j=1}^{\rho} \overrightarrow{\upsilon_{ji}} \Vert_2} \in  \OOmega_i$, translate $\mathscr{P}_i^{q_i}$ replacing $q_i$ for $q_i+\vartheta\dfrac{\sum_{j=1}^\rho \overrightarrow{\upsilon_{ji}}}{\Vert \sum_{j=1}^{\rho} \overrightarrow{\upsilon_{ji}} \Vert_2}$. \Comment{$\vartheta > 0$}
 \State \indent \indent $\bullet$ If $\ell_1(\mathscr{P}_i^{q_i},\mathscr{P}_j^{q_j})< 3(\kappa x +  \kappa y)$ for some $i,j\in \{1,...,\rho\}$ with $i\not=j$, go to STEP 3.
 \State \indent \indent $\bullet$ If this line has been revisited a number $\Upsilon_1\in \mathbb{N}$ of consecutive times without pass by a step different from STEP 3 and STEP
 \State \indent  4, go to STEP 5.
 \State \indent \indent $\bullet$  If STEP 4 has been revisited a number $\Upsilon_2\in \mathbb{N}$ of consecutive times without pass by another step, go to STEP 2. Otherwise,
  \State \indent  begin STEP 4 again.
  \State \indent  \textbf{STEP 5}  For each $i\in\{1,...,\rho-1\}$, let $T_i$ be the set of indices in $\{i+1,...,\rho\}$ such that $\ell_1(\mathscr{P}_i^{q_i},\mathscr{P}_j^{q_j})\geq 3(\kappa x +  \kappa y)$.
  \State \indent \indent $\bullet$ For each $i\in\{1,...,\rho-1\}$, if $T_i\not= \emptyset$, compute the point $q^*$ that solves the problem
  $$\max_{q\in \OOmega_i} \min_{j\in T_i} \ell_2(q,q_j).$$
  \State \indent \indent $\bullet$ For each $i\in\{1,...,\rho-1\}$, if $T_i\not= \emptyset$, replace $q_i$ by a point randomly selected in $\mathcal{B}_{\ell_1}(q^*,\varepsilon)$. \Comment{$\varepsilon >0$}
  \State \indent \indent Note: if $\mathcal{B}_{\ell_1}(q^*,\varepsilon)\cap \OOmega_i = \emptyset$, replace $q_i$ by the $\ell_1$-projection  of $q^*$ in $\OOmega_i$.
  \State \indent \indent $\bullet$ Randomly permute the order of the dimensional facilities $\{1,...,\rho\}$.
  \State \indent \indent $\bullet$ Go to STEP 1.
  \State \indent \textbf{STEP 6} Undo the possible permutations applied to the order of the dimensional facilities $\{1,...,\rho\}$ done in STEP 5, i.e., order the indices $\{1,...,\rho\}$ of the dimensional facilities as in the input.
  \State \indent \textbf{STEP 7} For each $i\in\{1,...,\rho\}$, determine the cell $(k_i,l_i)\in \bm{\Omega}$ to which $q_i$ belongs to.
  \State \indent \textbf{RETURN} $((k_1,l_1),...,(k_\rho,l_\rho))$.
  \State \textbf{END PROCEDURE}
  \end{algorithmic}
\end{algorithm}

The wavefront algorithm begins in STEP 1 with a shrunken version $\mathscr{P}_1^{q_1},...,\mathscr{P}_\rho^{q_\rho}$ (determined by parameter $\lambda$) of the sets $P_1^{q_1},...,P_{\rho}^{q_{\rho}}$. In STEP 2, if condition $\ell_1(\mathscr{P}_i^{q_i},\mathscr{P}_j^{q_j})\geq 3(\kappa x + \kappa y)$ is satisfied for all $i,j\in \{1,...,\rho\}$ with $i\not=j$, we can continue making to grow $\mathscr{P}_1^{q_1},...,\mathscr{P}_\rho^{q_\rho}$ applying them a homothecy of ratio $\lambda$. The meaning of the condition above is the following: the algorithm WAVE\_DIMFAC is able to find a suitable solution for problem BL from an initial location $(q_1,...,q_{\rho})\in \OOmega_1\times ... \times \OOmega_\rho$ (not necessarily feasible) of the dimensional facilities $P_1,...,P_\rho$; from this suitable solution of problem BL we will obtain a suitable solution of problem DBL moving each root point $p_i$ of the dimensional facilities from $q_i$ to the center of the cell of $\bm{\Omega}$ to which $q_i$ belongs to (STEP 7 of the algorithm); however, this movement may lead to some cases where  the interior of the cell facilities intersect, producing a non-suitable location of the facilities in problem DBL; it is easy to see that the condition above (onwards, the minimum $\ell_1$-separation-condition) avoids this undesirable situation in STEP 7. Since $1/\lambda\in \mathbb{N}$, it holds that $\{\mathscr{P}_1^{q_1},...,\mathscr{P}_{\rho}^{q_{\rho}}\} = \{P_1^{q_1},...,P_{\rho}^{q_{\rho}}\}$ after a finite number of homothecies of ratio $\lambda$ applied to sets $\mathscr{P}_1^{q_1},...,\mathscr{P}_{\rho}^{q_{\rho}}$ in STEP 2. If the minimum $\ell_1$-separation-condition  is not satisfied for some $i,j\in \{1,...,\rho\}$ with $i\not=j$ in STEP 2, we have to separate the pairs of problematic shrunken dimensional facilities.

The separation of the dimensional facilities $\mathscr{P}_1^{q_1},...,\mathscr{P}_\rho^{q_\rho}$ in STEP 3 and STEP 4 is done with the separator vectors $\overrightarrow{\upsilon}$. The separator vector $\overrightarrow{\upsilon_{ij}}$ gives the direction that moves away the root point $p_j=q_j$ of the dimensional facility $\mathscr{P}_j^{q_j}$ from the root point $p_i=q_i$ of the dimensional facility $\mathscr{P}_i^{q_i}$ such that does not verify the minimum $\ell_1$-separation-condition with respect to $\mathscr{P}_j^{q_j}$ (if the minimum $\ell_1$-separation-condition is verified then $\overrightarrow{\upsilon_{ij}}=\underline{0}$). Thus, $(\sum_{j=1}^{\rho} \overrightarrow{\upsilon_{ij}})/\Vert \sum_{j=1}^\rho \overrightarrow{\upsilon_{ij}} \Vert_2$ can be used as a direction to separate  $\mathscr{P}_j^{q_j}$ from the other sets which are too close to it. Parameter $\vartheta$ controls the distance of the separations. Separation steps are applied $\Upsilon_2\in \mathbb{N}$ times if all the pairs of dimensional facilities satisfy the minimum $\ell_1$-separation-condition in each iteration, otherwise, $\overrightarrow{\upsilon}$ has to be updated and the separation process has to begin again. So, the separation process ends when $\Upsilon_2$ iterations are done fixed $\upsilon$ or when an overall  number $\Upsilon_1\in\mathbb{N}$ of iterartions is reached. Note that this separation process  is especially effective when the root points of the dimensional facilities are chosen having some sort of centrality  meaning with respect to its shape, as the centroid or similar relevant points.

If the maximum number of iterations is reached (the third line of STEP 4 has been revisited a number $\Upsilon_1\in \mathbb{N}$ of consecutive times without pass by a step different from STEP 3 and STEP 4), we have to relocate in STEP 5 the root points of the dimensional facilities $\mathscr{P}_1^{q_1},...,\mathscr{P}_\rho^{q_\rho}$ which not satisfy the minimum $\ell_1$-separation-condition and begin the growing process again (from STEP 1). Root point $q_i$ is relocated maximizing the minimum Euclidean distance from the root points $q_j$ of dimensional facility $\mathscr{P}_j^{q_j}$ violating the minimum $\ell_1$-separation-condition with $\mathscr{P}_i^{q_i}$: found the solution $q^*$ of the problem $\max_{q\in \OOmega_i}\min_{j\in T_i} \ell_2(q,q_j)$, we relocate point $q_i$ at a point randomly selected in a neighbourhood of $q^*$ (we use the ball $\mathcal{B}_{\ell_1}(q^*,\varepsilon)$ as that neighbourhood). Actually, in STEP 5 of our algorithm, instead of solving  a global maximin problem, we solve a local maximin problem which needs a random point to start, making the process more random. As relocation of points $q_i$ done in STEP 5 depends on the order of the dimensional facilities, we then permute the order of the dimensional facilities to get more randomness in the algorithm. That pemutations has to be undone (STEP 6) before to determine the suitable solution of dimensional facilities $P_1,...,P_{\rho}$ found (STEP 7) and to return it.

\subsubsection{Greedy algorithm}

\indent Consider now that we are given a suitable solution $\bm{Q}\in \bm{\Gamma}$ for problem DBL. The greedy algorithm shown in Algorithm \ref{algGreedy} performs a local search to improve the objective value given by the current suitable solution $\bm{Q}$. Specifically, if $(k_i,l_i)\in\bm{\Omega}$ is the cell in whose center is located the root point $p_i$ of the dimensional facility $P_i$, the greedy algorithm evaluates the objective function of  problem DBL if we move  $p_i$ to the centers of the cells  in a neighbourhood of $(k_i,l_i)$ (determined by parameters $\Delta_k,\Delta_l \in \mathbb{N}$) keeping the position of the remaining root points, provided that the movement produces a suitable solution. This is done for each $i\in\{1,...,\rho\}$. Then, we relocate the dimensional facilities whose movement to a neighbor cell provides the best improvement of the objective value. This process is repeated until no improvement is obtained.

\begin{algorithm}
  \caption{Greedy algorithm to improve a suitable solution of problem DBL} \label{algGreedy}
  \label{algGreedy}
  \begin{algorithmic}[]
  \State \textbf{PROCEDURE} GREEDY\_DIMFAC($((k_1,l_1),...,(k_{\rho},l_{\rho}))$)
  \State \indent \textbf{STEP 1} For each $i\in\{1,...,{\rho}\}$, compute the best improvement in the objective value of problem DBL given by a suitable solution of the form $((k_1,l_1),...,(k_i+j_k, l_i+j_l),...,(k_{\rho},l_{\rho}))$, for any $j_k\in\{-\Delta_k,..., \Delta_k\}$ and $j_l\in\{-\Delta_l,..., \Delta_l\}$.
  \State \indent \indent $\bullet$ If no improvement is achieved for all $i\in \{1,...,\rho\}$, then $((k_1,l_1),...,(k_{\rho},l_{\rho}))$ is the suitable solution obtained by the
  \State \indent  algorithm. Go to RETURN.
  \State \indent \indent $\bullet$ Otherwise, go to STEP 2.
  \State \indent \textbf{STEP 2} If $i^*$ is the index which provides the best improvement in the objective value of  problem DBL in STEP 1 among all the indices in  $\{1,...,\rho\}$, replace $((k_1,l_1),...,(k_{\rho},l_{\rho}))$ by $((k_1,l_1),...,(k_{i^*}+j_k^*, l_{i^*}+j_l^*),...,(k_{\rho},l_{\rho}))$, being $j_k^*$ and $j_l^*$ the values of $j_k$ and $j_l$ that give the best improvement for index $i^*$. Go to STEP 1.
  \State \indent \textbf{RETURN} $((k_1,l_1),...,(k_{\rho},l_{\rho}))$.
  \State \textbf{END PROCEDURE}
  \end{algorithmic}
\end{algorithm}

Note that Algorithm \ref{algGreedy} is presented for a grid $G$ where the neighbors of a cell $(r,s)\in\bm{G}$ are determined by the adjacent horizontal and vertical cells in $\bm{G}$. This is done for the sake of simplicity. However, it is easy to extend the Algorithm \ref{algGreedy} to more general grids if the neighborhood of a cell is well defined in the considered grid.

\subsection{Computational experiments}

This section reports some computational experiments performed to show the usefulness  of the proposed methodologies to solve problem BL. Our code is implemented in MATLAB R2017A and  to solve the MILP programs it makes calls to the XPRESS solver version 8.0. All experiments were run in a computer DellT5500 with a processor Intel(R) Xeon(R) with a CPU X5690 at 3.75 GHz and 48 GB of RAM memory.

We have included several test examples. Some of them were already proposed in \cite{optPartition} and some others are new. Including the new examples we want to compare the diversity of the solutions when different utilities, distance measures, shapes of the dimensional facilities, cost functions and densities are combined. In addition, we also show how the solutions of the examples are affected when they are included in the bilevel approach combining  the different elements of the problem.

In all cases, we use regular grids to approximate the exact solution of the bilevel problem BL. We always begin by solving the problem by means of our heuristic algorithm (Algorithm \ref{algGRASP}) with the following parameters. Algorithm \ref{algGRASP} runs with a list of solutions of length $\psi=50$, $ \varpi=2$ root points to be permuted and the stopping criterion, in STEP 2, consists of processing the list without improvement. Algorithm \ref{algWave} is executed with a homothecy ratio $\lambda=0.05$ (which results in applying the hotothecy transformation, at least, $1/\lambda=20$ times), a separation parameter $\vartheta=0.05$ and stopping separation criterion $\Upsilon_1=9$ and $\Upsilon_2=3$. Finally, Algorithm \ref{algGreedy} is applied with $\Delta_k=\Delta_l=5$. Once the heuristic solution is found, we next improve that solution adding it as initial feasible solutions to the exact MILP formulation and then we let it run for 4 hours (14400s) of CPU time. The performance of the GRASP heuristic and the MILP formulation is reported in Table \ref{table1}. This table shows, for both methods, the time required for the preproceessing of the information   (PT), the execution time once the information has been preprocessed  (ET), the best objective value found (BOVF) and  the gap obtained for the solution provided by the MILP formulation (GAP).

\bigskip

\noindent \emph{Example 1} \,\,Our first test  illustrates how the approach in this paper applies to one example borrowed from the literature \cite{optPartition}. First, we consider that the demand region $\Omega$, the dimensional facilities $P_1,P_2,P_3$ and all the elements of the lower level problem are the ones given in Example 4.1 in \cite{optPartition}. In addition, we will assume that there also exists an installation cost described by the base installation cost density
$B(q) = 6(x-y)$ if $x\geq y$ and $B(q) = 0$ otherwise, and the installation cost functions $I_i(\omega_i^I) = \omega_i^I$ for all $\omega_i^I\in[0,1]$.

To better illustrate the performance of our methodology, we distinguish two different situations.

\bigskip

\noindent \emph{Example 1.1} \,\,The first situation includes non-uniform demand density $D$ on $\Omega$, given by
$D(q) = 8 (x-0.5)$ if $x\geq 0.5$ and $D(q)=0$ otherwise.
In addition, we also consider the following non-zero lost demand cost $L(\omega^L)=\omega^L$ for all $\omega^L\in [0,1]$ and zero congestion costs $C_i(\omega_i^C) = 0$ for all $\omega_i^C\in[0,1]$, for each $i\in\{1,...,3\}$. Note that, in this example, the problem does not explicitly depend on the partition of the demand region.

\bigskip

\noindent \emph{Example 1.2} \,\,The second example considers uniform demand density $D(q)=1$, as in the original example in \cite{optPartition}, it does not apply any lost demand cost (i.e., $L(\omega^L)=0$ for all $\omega^L\in [0,1]$) but it includes the following congestion costs:
$ C_i(\omega_i^C) = \omega_i^C$ if $\omega_i^C< \tilde{\omega}_i^C/3$ and $C_i(\omega_i^C)=\omega_i^C + 100(\omega_i^C-\tilde{\omega}_i^C/3)$ otherwise, for all $\omega_i^C\in[0,1]$ and each $i\in\{1,...,3\}$, being $\tilde{\omega}_i^C=1-\int_{\bigcup_{i=1}^3 P_i} D(q)dq$. The inclusion of this congestion cost term makes the problem to depend on the partition of the demand region. The choice of this particular expression forces an approximate equal splitting of the demand among the three facilities.

We have solved the location-allocation problems defined by these situations and the results can be seen in Fig. \ref{example1} and Table \ref{table1}.


\begin{figure}
\centering
\begin{minipage}[!ht]{0.45\textwidth}
\centering
\includegraphics[scale=0.5]{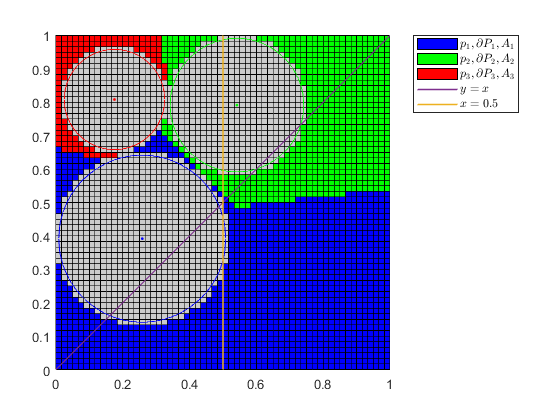}\caption*{\small (a) Heuristic solution of Example 1.1}
\end{minipage}
\begin{minipage}[!ht]{0.45\textwidth}
\centering
\includegraphics[scale=0.5]{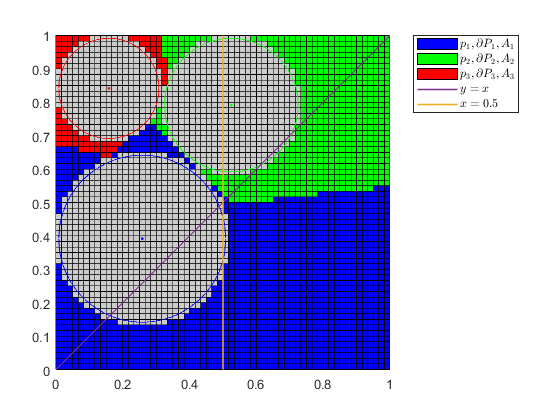}\caption*{\small (b) Exact solution of Example 1.1}
\end{minipage}\\
\begin{minipage}[!ht]{0.45\textwidth}
\centering
\includegraphics[scale=0.5]{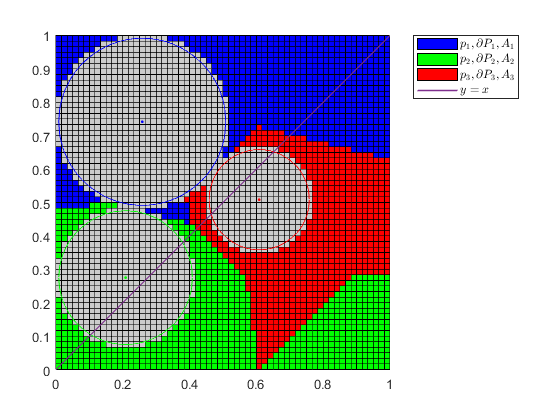}\caption*{\small (c) Heuristic solution of Example 1.2}
\end{minipage}
\begin{minipage}[!ht]{0.45\textwidth}
\centering
\includegraphics[scale=0.5]{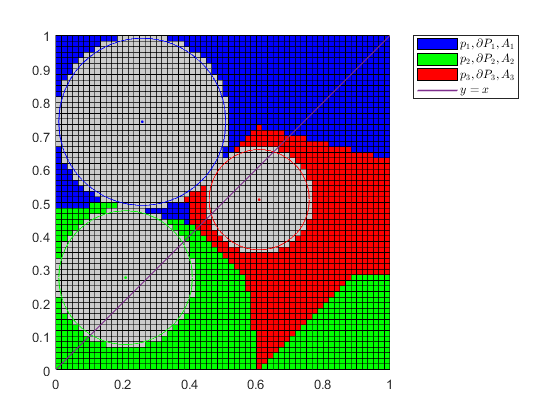}\caption*{\small (d) MILP solution of Example 1.2}
\end{minipage}
\caption{Graphical output of the solutions for Example 1 \label{example1}}
\end{figure}

The graphical output of our algorithms reports the results that could have been anticipated. In the Example 1.1 (Fig. \ref{example1}.(a)-(b)), since the base installation cost density is null in the upper triangle and the demand is also null in the left half of the region $\Omega$, the dimensional facilities tend to be located, as much as possible, in the upper triangle and in the left half of the square region. However, their measure does not allow them to be completely included in that region. This is the reason why two of them have a portion on the higher installation cost (lower triangle) and higher lost demand (right half square) parts of the diagram. The MILP formulation provides a solution (Fig. \ref{example1}.(b)) slightly better than the one obtained by the GRASP heuristic  (Fig. \ref{example1}.(a)). The reader should observe that the solution provided by the MILP formulation is optimal (it has zero GAP) as it can be seen in Table \ref{table1}.

The solution of the location-allocation problem of Example 1.2 is shown in Fig. \ref{example1}.(c)-(d). The result shown in these figures is consistent and it shows that the installation cost does not fully determine the final location of the dimensional facilities. This can be explain because, a non approximate equal splitting of the demand among the facilities, is highly penalized by the installation cost functions $I_1,I_2,I_3$. Even so, the solution attempts to place the facilities in the upper triangle to also reduce the installation cost, actually, the biggest facility is completely contained in the upper triangle. The solution obtained by the MILP formulation (Fig. \ref{example1}.(d)) is the same that the one provided by the GRASP heuristic  (Fig. \ref{example1}.(c)). This means that, in this case,  the MILP formulation is not able to find a better solution than the GRASP heuristic within the CPU time limit. However, the use of the MILP approach provides the GAP of the solution obtained (see Table \ref{table1}). In Fig. \ref{example1}.(c)-(d). Finally, in Fig. \ref{example1}(c)-(d), it seems that the demand region $A_2$, assigned to the second facility $P_2$, has two connected components. This fact is not strange if one has in mind the properties of bisectors for different distance measures, as it is our case, see \cite{icking} and \cite{NP05} for more details.

\bigskip

\noindent \emph{Example 2} \,\,This situation is included to illustrate the use of different utilities in the model. We consider that the demand region $\Omega$ is the unit square and there are three dimensional facilities. The first one,
$P_1$,  is a non-convex polygon with utility based on a conservative planner given by $\displaystyle u_1(q,q_1) = 0.8 \max_{\tilde{q}\in P_1^{q_1}} \ell_2(q-\tilde{q})$.
The second facility, $P_2$, is a regular pentagon and its utility,  is given by $u_2(q,q_2) = \ell_2(q-q_2)$, where $q_2$ is the centroid of $P_2$. Finally, the last facility, $P_3$, is the unit ball of a weighted Euclidean norm, namely $P_3:=\{(x,y): \sqrt{75 x^2 + 150 y^2}\le 1\}$; and its utility $u_3(q,q_3) = 0.2\gamma_{P_3}(q-q_3)$. This is the case where the utility is induced by a Minkowski functional.

The remaining parameters of this example are the following: $a_1=a_2=a_3=1$ and the demand density is uniform, namely $D(q) = 1$. The congestion costs, $C_i$ are: $C_i(\omega^C)=\omega^C$ if $\omega^C\le 0.25$ and $C_i(\omega^C)=\omega^C+\aleph_i\omega^C$  if $\omega^C\ge 0.25$, where $\aleph_1=\aleph_2=0.5$ and $\aleph_3=0.25$;  and the lost demand cost is $L(\omega^L)=\omega^L$.

Finally, the base installation cost density is defined by the expression $$B(q) = \begin{cases} 2(x+y), & \hbox{if $x\leq 0.5$ and $y\leq 0.5$,}  \\ 2(x+1-y), & \hbox{if $x\leq 0.5$ and $y > 0.5$,}  \\ 2(1-x+y), & \hbox{if $x> 0.5$ and $y\leq 0.5$,}  \\ 2(1-x+1-y), & \hbox{if $x> 0.5$ and $y> 0.5$.}  \end{cases}$$
This function accumulates the density in the center of the square since the bivariate density function increases from the vertices of the unit square to its center. We take as base installation costs $I_1(\omega^I) = I_2(\omega^I) = I_3(\omega^I) = 5\omega^I$.

We solve this configuration for grids with different sizes to illustrate the convergence of our discretization approach. We have chosen grids of $20\times 20$ (see Fig. \ref{example2}(a)-(b)), $30\times 30$ (see Fig. \ref{example2}(c)-(d)), $40\times 40$ (see Fig. \ref{example2}(e)-(f)), $50\times 50$ (see Fig. \ref{example2}(g)-(h)) and $60\times 60$ (see Fig. \ref{example2}(i)-(j)).

From our results we report that in all cases (i.e., for the different grid sizes) the exact MILP approach could not improve the solution found by our heuristic algorithm. In the $20\times 20$ grid case the solution found is optimal, as certified by the MILP problem (see Table \ref{table1}). The configuration of the solutions found can be explained by the shapes of the densities and costs functions. Since the base installation density is lower in the vicinity of the vertices of $\Omega$ and the demand is uniform, the facilities try to locate the closer to the vertices the better. Nevertheless, the congestion cost makes that one of the facilities that is less congested, $P_3$,  moves closer to $P_1$ to cannibalize part of its demand. The two connected components in the partition allocated to $P_1$ in Fig. \ref{example2}(i)-(j) can be explained, as before, by the properties of bisectors with different norms. Finally, one observes some stability in the solutions whenever the grids are denser.

\begin{figure}
\centering
\begin{minipage}[!ht]{0.45\textwidth}
\centering
\includegraphics[scale=0.35]{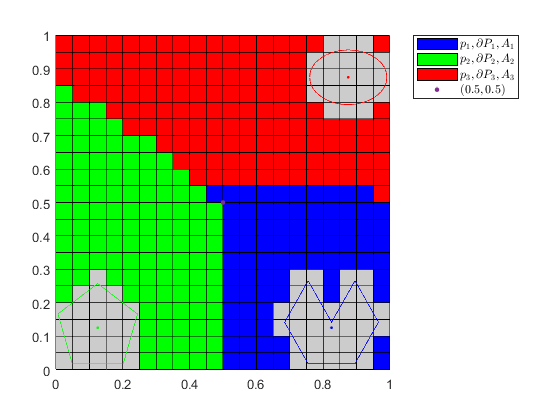}\caption*{\small (a) Heuristic solution of Example 2 (20$\times$20)}
\end{minipage}
\begin{minipage}[!ht]{0.45\textwidth}
\centering
\includegraphics[scale=0.35]{Ex22020H.png}\caption*{\small (b) Exact solution of Example 2 (20$\times$20)}
\end{minipage}\\
\begin{minipage}[!ht]{0.45\textwidth}
\centering
\includegraphics[scale=0.35]{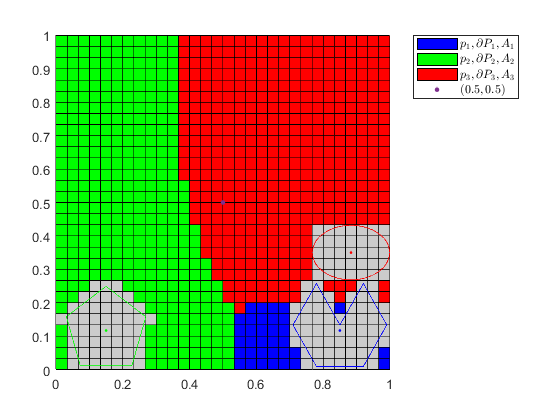}\caption*{\small (c) Heuristic solution of Example 2 (30$\times$30)}
\end{minipage}
\begin{minipage}[!ht]{0.45\textwidth}
\centering
\includegraphics[scale=0.35]{Ex23030H.png}\caption*{\small (d) MILP solution of Example 2 (30$\times$30)}
\end{minipage}
\begin{minipage}[!ht]{0.45\textwidth}
\centering
\includegraphics[scale=0.35]{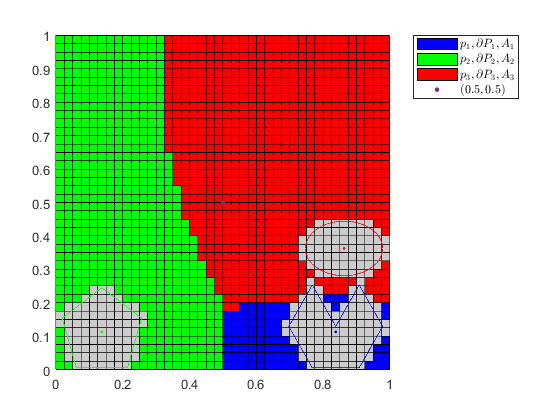}\caption*{\small (e) Heuristic solution of Example 2 (40$\times$40)}
\end{minipage}
\begin{minipage}[!ht]{0.45\textwidth}
\centering
\includegraphics[scale=0.35]{Ex24040H.png}\caption*{\small (f) MILP solution of Example 2 (40$\times$40)}
\end{minipage}
\begin{minipage}[!ht]{0.45\textwidth}
\centering
\includegraphics[scale=0.35]{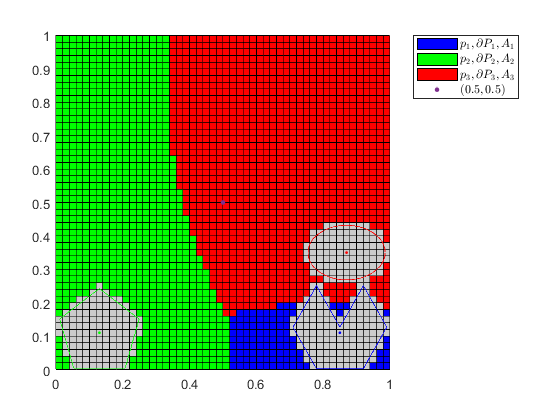}\caption*{\small (g) Heuristic solution of Example 2 (50$\times$50)}
\end{minipage}
\begin{minipage}[!ht]{0.45\textwidth}
\centering
\includegraphics[scale=0.35]{Ex25050H.png}\caption*{\small (h) MILP solution of Example 2 (50$\times$50)}
\end{minipage}
\begin{minipage}[!ht]{0.45\textwidth}
\centering
\includegraphics[scale=0.35]{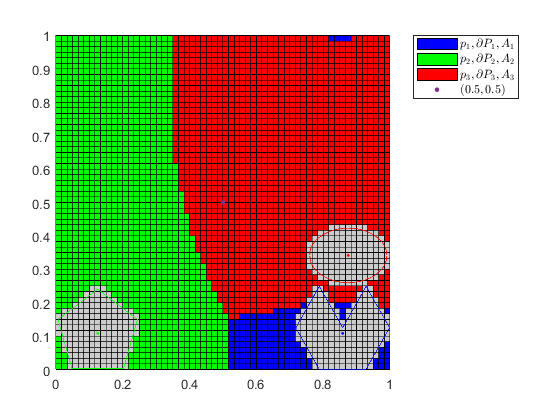}\caption*{\small (i) Heuristic solution of Example 2 (60$\times$60)}
\end{minipage}
\begin{minipage}[!ht]{0.45\textwidth}
\centering
\includegraphics[scale=0.35]{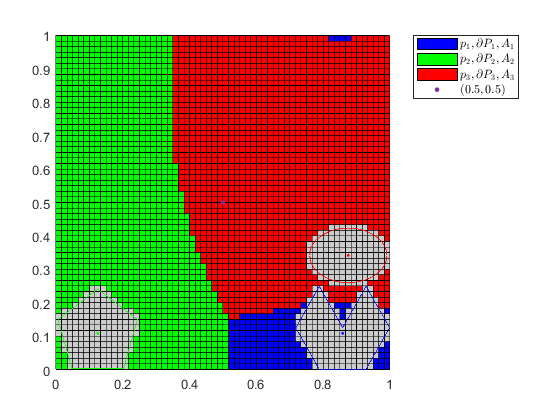}\caption*{\small (j) MILP solution of Example 2 (60$\times$60)}
\end{minipage}
\caption{Graphical output of the solutions for Example 2 \label{example2}}
\end{figure}

\bigskip

\noindent \emph{Example 3} \,\,This situation is included to illustrate our methodology with a larger number of facilities, 10, and also with different congestion costs associated to each of them. We assume an utility $u_i(q,q_i)=\ell_ 2(q-q_i)$, for all $i=1,\ldots,10$. Once again, we consider that the demand region $\Omega$ is the unit square and there are ten tetrominoes as dimensional facilities. $\Omega$ is discretized in a $60\times 60$ grid.  Since, our tetrominoes have small measure with respect to $\Omega$, this example considers that the installation and lost demand costs are negligible and thus we take them as null.

We report two examples that depend on different demand density functions and in both cases the congestion costs for the tetrominoes are the same. These congestion costs are: $C_i(\omega^C)=\omega^C$ if $\omega^C\le \aleph_i$ and $C_i(\omega^C)=\aleph_i+100(\omega^C-\aleph_i)$  if $\omega^C\ge \aleph_i$ where $\aleph=1/30(5,5,3,3,3,3,3,3,1,1)$.
\bigskip

\noindent \emph{Example 3.1} \,\,In this case, we have chosen a uniform demand density $D(q)=1$.

\bigskip

\noindent \emph{Example 3.2} \,\,This second case considers as demand density for each $q=(x,y)\in \Omega$,
$D(q) =  3(x-y)$ if $x\geq y$, and  $D(q) =3(y-x)$, if $x<y$. Observe that this density is null on the diagonal of $\Omega$ and is maximal at the points $(1,0)$ and $(0,1)$.

As it can be seen in Fig. \ref{example3}, the results provided by the two methods, in the two cases (Example 3.1 and 3.2), are consistent. As it occurs in Example 2, the exact MILP approach could not improve the solution found by our GRASP algorithm in both cases. Analyzing the results in each case, we observe that in Example 3.1 (Fig. \ref{example3} (a)-(b)) the facilities are spread more or less ``uniformly'' on the unit square. This is due to the considerd uniform demand density. On the other hand, in Example 3.2 (Fig. \ref{example3} (c)-(d)) the facilities are mainly concentrated close to the points $(1,0)$ and $(0,1)$, where the demand density is much higher. It is also interesting to remark the excellent behavior of our MILP formulation in Example 3.1. We observe in Table \ref{table1} that, given the structure of the problem defined in  Example 3.1 (there is no installation cost, no lost demand cost, and all the cells have the same demand density), the MILP formulation is able to prove  optimality (GAP (\%) is zero) of the solution found by the GRASP heuristic in a rather short computing time.
\bigskip

\begin{figure}
\centering
\begin{minipage}[!ht]{0.45\textwidth}
\centering
\includegraphics[scale=0.5]{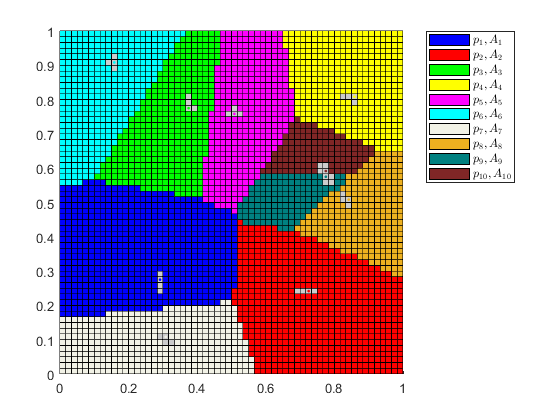}\caption*{\small (a) Heuristic solution of Example 3.1}
\end{minipage}
\begin{minipage}[!ht]{0.45\textwidth}
\centering
\includegraphics[scale=0.5]{Ex316060H.png}\caption*{\small (b) Exact solution of Example 3.1}
\end{minipage}\\
\begin{minipage}[!ht]{0.45\textwidth}
\centering
\includegraphics[scale=0.5]{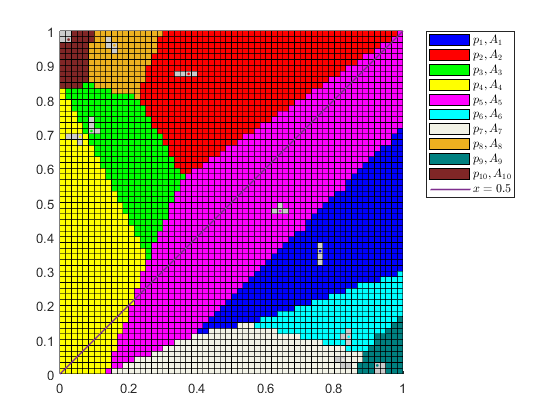}\caption*{\small (c) Heuristic solution of Example 3.2}
\end{minipage}
\begin{minipage}[!ht]{0.45\textwidth}
\centering
\includegraphics[scale=0.5]{Ex326060H.png}\caption*{\small (d) MILP solution of Example 3.2}
\end{minipage}
\caption{Graphical output of the solutions for Example 3 \label{example3}}
\end{figure}

\begin{table}[htb]
\centering
\begin{tabular}{|l|l|l|l|l|l|l|l|l|}
\hline
\multicolumn{2}{|c|}{Method} &  \multicolumn{3}{c|}{GRASP heuristic} & \multicolumn{4}{c|}{MILP formulation}\\ \cline{3-9} \hline
Example & Grid & PT (s) & ET (s) & BOVF & PT (s) & ET (s) & BOVF & GAP (\%)\\
\hline \hline
1.1 & $60\times 60$ & 11336 & 515 & 0.1109 & 469 & 3306 & 0.0995 & 0.00\\ \cline{1-9}
1.2 & $60\times 60$ & 11496 & 350 & 0.6479 & 472 & 14400 & 0.6479 & 8.89\\ \cline{1-9}
2 & $20\times 20$ & 52 & 101 & 1.4594 & 48 & 1209 & 1.4594 & 0.00 \\ \cline{1-9}
2 & $30\times 30$ & 331 & 193 & 1.3236 & 249 & 14400 & 1.3236 & 20.56 \\ \cline{1-9}
2 & $40\times 40$ & 1422 & 364 & 1.3012 & 724 & 14400 & 1.3012 & 22.26 \\ \cline{1-9}
2 & $50\times 50$ & 5261 & 524 & 1.2451 & 1845 & 14400 & 1.2451 & 22.80 \\  \cline{1-9}
2 & $60\times 60$ & 14333 & 721 & 1.2123 & 3817 & 14400 & 1.2123 & 25.24 \\ \cline{1-9}
3.1 & $60\times 60$ & 108802 & 55356 & 3.1889 & 4670 & 450 & 3.1889 & 0.00 \\ \cline{1-9}
3.2 &  $60\times 60$ & 110915 & 25303 & 2.1139 & 4824 & 14400 & 2.1139 & 42.97 \\ \cline{1-9}
\end{tabular}
\caption{GRASP heuristic and MILP formulation performance}
\label{table1}
\end{table}

\section{Conclusions}

This paper gives a first complete proof of existence of optimal solutions of a general location-allocation problem with dimensional facilities. This result includes as particular instances previously published results in the field with dimensionless facilities (point facilities). It also provides two  methods to solve this problem using sequences of solutions for a discrete approximation of the problem. One is exact and it is based on a new mixed-integer linear programming formulation and the other one is a GRASP heuristic that results in very good solutions.

This paper has a  number of possible extensions that may open some interesting research lines. Among them, we would like to mention relaxing some conditions ensuring existence of optimal solutions, as for instance the continuity of the utilities in the objective function of the lower level problem, although this is beyond the scope of this paper. In addition, these results can be extended to any finite dimension space at the price of  increasing the complexity of the discrete models that then become exponential in the dimension of the space.

\begin{acknowledgements}
This paper was originated during a visit of Prof. L. Mallozzi at the University of Seville supported by the PhD Program Mathematics. The authors want to thanks Prof. A. Lewis for his suggestion to tackle the  general location-allocation problem using a discretization scheme suggested during a presentation of this material in a seminar given during the previously mentioned visit. Finally, we would also like to thank the Ministry of Economy and Competitiveness of Spanish Government for partially funding our research via project MTM2016-74983.
\end{acknowledgements}


\end{document}